\documentclass[leqno,12pt]{amsart}
\usepackage{amssymb,amsthm,amsmath,txfonts,a4wide}
\usepackage{xcolor}
\usepackage{enumerate}
\usepackage{hyperref}
\usepackage{tikz}
\usetikzlibrary{shapes,arrows}
\usepackage{times}
\usepackage{verbatim}

\hypersetup{
colorlinks,
citecolor=green,
filecolor=blue,
linkcolor=red,
urlcolor=blue
}

\newtheorem{defn}{Definition}[section]
\newtheorem{lemma}[defn]{Lemma}
\newtheorem{remark}[defn]{Remark}

\newtheorem{theorem}[defn]{Theorem}
\newtheorem{corollary}[defn]{Corollary}

\newtheorem{definition}[defn]{Definition}
\newtheorem{proposition}[defn]{Proposition}

\newtheorem{claim}[defn]{Claim}

\numberwithin{equation}{section}

\begin{document}
	
\title{Scattering property for a system of Klein-Gordon equations with energy below ground state}
\author{Yan Cui}
\thanks{Yan Cui is supported by Funding by Science and Technology Projects in Guangzhou (No. 2023A04J1335) and the Natural Science Foundation of China (No. 12001555). }
\address{Department of Mathematics, Jinan University, Guangzhou, P. R. China.}
\email{cuiy32@jnu.edu.cn}

\author{Bo Xia}
\address{School of Mathematical Sciences, USTC, Hefei, P. R. China.}
\email{xiabomath@ustc.edu.cn, xaboustc@hotmail.com}
\thanks{Bo Xia was supported by NSFC No. 12171446.}

\begin{abstract}
	In the previous work \cite{Cui2022}, we classified the solutions to a family of systems of Klein-Gordon equations with non-negative energy below the ground state into two parts: one blows up in finite time while the other extends to a global solution. In the present work, we strengthen this result, showing that these global solutions are indeed scattering in the energy space. Here we adapted Kenig-Merle's concentration-compactness approach to the system.
\end{abstract}
\maketitle	
\tableofcontents


\section{Introduction}
	In this article, we consider the following system of Klein-Gordon equations
	\begin{equation}\label{eq:skg:int}
	\left\{
	\begin{split}
	 \partial^2_{tt} u_1-\Delta u_1 + u_1 &=u^3_1 + \beta u_2^2u_1, &\mbox { in }& \mathbb{R}\times\mathbb{R}^3,\\
	 \partial^2_{tt} u_2-\Delta u_2 + u_2 &=u^3_2 + \beta u_1^2u_2, &\mbox { in } &\mathbb{R}\times\mathbb{R}^3,\\
	(u_1(0),\partial_tu_{1}(0))&=(u_{10},u_{11}),&\mbox { in }& \mathbb{R}^3, \\
	 (u_2(0),\partial_tu_{2}(0))&=(u_{20},u_{21}),&\mbox { in }&\mathbb{R}^3,
	\end{split}
	\right.
	\end{equation}	
	where $\beta\in{\mathbb{R}}$ is a parameter.	
	Some other system of this kind was introduced by Segal in \cite{segal1965} to model the motion of charged mesons in electromagnetic field. Similar systems of coupled wave and Klein-Gordon equations are also proposed to model some physical phenomena such as the interaction of mass and massless classical fields, and long longitudinal waves in elastic bi-layers, see \cite{georgiev1990,  ionescu2019, 2007Coupled} and references therein for more details.
	
	The system \eqref{eq:skg:int} has also been mathematically studied, see \cite{georgiev1990,Cui2022}. In particular, in our previous work, we classified the datum with positive energy but below the ground state into two classes: one leads to global solutions while the other leads to blowing up solutions.  {The aim of present study is to establish the scattering properties for such long time solutions.}
	
	In order to state our results, we introduce the following two functionals defined on $H^1\times H^1$
	\begin{equation*}
	J[\varphi_1,\varphi_2]:=\frac{1}{2}\int_{\mathbb{R}^3}\sum_{j=1}^2\left[|\nabla\varphi_j|^2+|\varphi_j|^2\right]\mathrm{d}x-\frac{1}{4}\int_{\mathbb{R}^3}\left[\varphi_1^4+\varphi_2^4+2\beta\varphi_1^2\varphi_2^2\right]\mathrm{d}x
	\end{equation*}
	and
	\begin{equation*}
	K_0[\varphi_1,\varphi_2]:=\int_{\mathbb{R}^3}\sum_{j=1}^2\left[|\nabla\varphi_j|^2+|\varphi_j|^2\right]\mathrm{d}x-\int_{\mathbb{R}^3}\left[\varphi_1^4+\varphi_2^4+2\beta\varphi_1^2\varphi_2^2\right]\mathrm{d}x
	\end{equation*}
	For each fixed $\beta\in[0,\infty)$, Sirakov showed that the minimum
	\begin{equation}\label{eq:inf}
		\inf\left\{J[\varphi_1,\varphi_2]:(\varphi_1,\varphi_2)\in \left(H^1\times H^1\right)\backslash\{(0,0)\},\ K_0[\varphi_1,\varphi_2]=0\right\}.
	\end{equation}	
	is assumed by some nonzero element $(Q_1,Q_2)\in H^1\times H^1$ and the minimum $J[Q_1,Q_2]$ is a positive quantity. Given the quandruple $\vec{U}:=\left((u_1,v_1),(u_2,v_2)\right)^{\mathsf{T}}\in (H^1\times L^2)^2$, we denote
	$$
	E[\vec{U}]:=\frac{1}{2}\int_{\mathbb{R}^3}\sum_{i=1}^2\left[|u_i|^2+|\nabla_{x} u_i|^2+|v_i|^2\right]\mathrm{d}x-\frac{1}{4}\int_{\mathbb{R}^3}\left[u_1^4+u_2^4+2\beta u_1^2u^2_2\right]\mathrm{d}x.
	$$
	
	In our previous work \cite{Cui2022}, we obtained
	\begin{theorem}\label{thm:1} With the notations as above and $\mathcal{H}:=H^1\times L^2$, {for each $\beta\in[0,\infty)$},  both regions defined by
		\begin{equation*}
			\mathcal{PS}^+:=\left\{\vec{U}\in\mathcal{H}\times\mathcal{H}:E[\vec{U}]<J[Q_1,Q_2],K_0[U]\geq 0\right\}
		\end{equation*}
		and
		\begin{equation*}
			\mathcal{PS}^-:=\left\{\vec{U}\in\mathcal{H}\times\mathcal{H}:E[\vec{U}]<J[Q_1,Q_2],K_0[U]< 0\right\}
		\end{equation*}
		are invariant under the flow of \eqref{eq:skg:int}, as long as the flow is defined. What's more, arguing in the spirit of Payne-Sattinger yields the following dichotomy:
		\begin{itemize}
			\item {the solution issued from the data in $\mathcal{PS}^+$ exists for all time.}
			\item {the solution issued from the data in $\mathcal{PS}^-$ blows up in finite time.}
		\end{itemize}
	\end{theorem}
	
	In the present work, we will strengthen the above regult, showing that any global solution $\vec{U}(t)$ issued from $\mathcal{PS}^+$ shares much finer property: it is scattering in the sense that for some $\vec{V}_{\pm}\in\mathcal{H}\times\mathcal{H}$
	\begin{equation*}
		\lim_{t\rightarrow\pm\infty}\left\|\vec{U}(t)-\vec{S}(t)\vec{V}_{\pm}\right\|_{\mathcal{H}\times\mathcal{H}}=0
	\end{equation*}
	where $\vec{S}(t)$ is the operator of free evolution, see the definition in \eqref{eq:free:evo}.
	
	\begin{theorem}\label{thm:main}Under the same assumption as in Theorem \ref{thm:1}, any long time solution issued from the data in $\mathcal{PS}^+$ scatters.
	\end{theorem}
	\begin{remark}{
	It is easy to extend results in $\mathbf{Theorem}$ \ref{thm:main} to the following general system
	\begin{equation}\label{eq:skg:int0}
	\left\{
	\begin{split}
	\partial_t^2u_{1}-\Delta u_1+ u_1&=\mu_1u_1^3+\beta u_2^2u_1, \mbox { in } \mathbb{R}\times\mathbb{R}^3,\\
	\partial_t^2u_{2}-\Delta u_2+ u_2&=\mu_2u^3_2+\beta u^2_1u_2, \mbox { in } \mathbb{R}\times\mathbb{R}^3,
	\end{split}
	\right.
	\end{equation}	
	where $\mu_1,\mu_2>0$ and $\beta\geq0$. Since the corresponding variational characterization of ground state solution
of \eqref{eq:skg:int0} (see \cite[Proposition 3.5]{sirakov07}) has the lowest mountain pass level being positive, one has no difficulty in going through our proof for \eqref{eq:skg:int0}.
	
	As far as the authors know, the more general systems  have been studied in \cite{chenzou2012,chenzou2015,pengwangwang2019,soavetavares2016,tavares2011} and \cite{weiwu2020}, where nontrivial solutions are proven to exist. Thus as long as one can show $h_0$ has a constrained characterization, we might expect to extend the results in $\mathbf{Theorem}$ \ref{thm:main} to the corresponding system of Klein-Gordon equations, by going through our argument in present article.
}
    \end{remark}
	\begin{remark}\label{rem:1}
		For systems of Schr\"{o}dinger equations,
		the scattering property below the ground state has already been considered in \cite{xu2014,Masaki2023,Xia2019}. We emphasize here that in the first two works the authors adopted Kenig-Merle's concentration-compactness argument \cite{Kenig2008} while in the third one the authors used Dodson-Murphy's interactive Morawetz estimate \cite{Dodson2018}. Although our proof of Theomem \ref{thm:main} falls into Kenig-Merle's framework, considering these two different approaches, we also expect the scattering property below the ground state energy for \eqref{eq:skg:int} should be achieved via the interactive Morawetz estimates.
	\end{remark}

	As is mentioned in Remark \ref{rem:1}, our proof of Theorem \ref{thm:main} runs in the same lines as in \cite{nakanishischlag2011} and \cite{Kenig2008}, whose outline we now turn to sketch. The whole argument is indeed by absurd, initially assuming that there exists a nonnegative number $E_\ast<J[Q_1,Q_2]$, for which we could find a sequence $\left\{\vec{U}_n\right\}\subset \mathcal{PS}^+$, satisfying
	\begin{equation*}
		E[\vec{U}_n]\nearrow E_\ast {\geq0}
	\end{equation*}
	and
	\begin{equation*}
		\left\|U_n\right\|_{\left(L^3_t(\mathbb{R},L^6_x(\mathbb{R}^3))^2\right)^2}\rightarrow_{n\rightarrow\infty}\infty.
	\end{equation*}
	Thanks to the small-data theory (see the fifth item in Proposition \ref{prop:cauchy}), $E_\ast$ is indeed a positve number. The strategy to derive contradiction is given as follows.
	\begin{enumerate}
		\item To extract a critical element $\vec{U}_\ast$ with the following property: there exists $(t_n,x_n)\in\mathbb{R}\times\mathbb{R}^3$ satisfying for each $n$
		\begin{equation*}
			\vec{U}_n=\vec{U}_\ast(\cdot+t_n,\cdot+x_n)+\vec{r}_n
		\end{equation*}
		where
			(a) $U_\ast$ is a strong solution to \eqref{eq:skg:int};
			(b) $\|\vec{r}_n\|_{L^\infty_t(\mathcal{H}\times\mathcal{H})}\rightarrow_{n\rightarrow\infty}0$;
			(c) $\left\|U_\ast\right\|_{\left(L^3_tL^6_x\right)^2}=\infty$;
			(d) $K_0[U_\ast]\geq 0$. This is achieved by combining several technical tools: linear profile decomposition, nonlinear profile decomposition and the perturbation lemma, together with the variational characterization of the ground state solution.
		\item To show the compactness property of the critical element. Precisely, for some vector-valued function $x_0:\mathbb{R}\ni t\mapsto x_0(t)\in\mathbb{R}^3$, both of sets
		\begin{equation*}
			\mathcal{K}_{\pm}:=\left\{ \vec{U}_\ast(t,x+x_0(t)):0\leq \pm t<\infty\right\}
		\end{equation*}
		are precompact in $\mathcal{H}\times\mathcal{H}$. The key result in this step is Proposition \ref{lemma:2}, roughly stating that at each time $t\geq0$, the linear energy of the critical element is uniformly bounded in the infinite region sufficiently away from $x_0(t)$.
		\item To show the $0$-momentum property of the critical element. This is a corollary of the relationship between energy and moment for solutions to \eqref{eq:skg:int} under Lorentz transform (see Proposition \ref{prop:lorentz}) and the dichotomy of the flows associated to \eqref{eq:skg:int} (see Theorem \ref{thm:1}).
		\item To improve the growth of $x_0(t)$. This is a direct corollary of the fact that time variation of the properly truncated center of energy is under control of the free energy away from the origin, which in turn follows from the local conservation law and the $0$-momentum of the critical element.
		\item To kill the critical element. This is achieved by considering the time variation of the localized product of the spontaneous derivative of the critical element and the action of antisymmetric dilation generator on it up to the time $t_0$, at which the growth of $x_0(t)$ is improved. On the one hand, we can bound this variation from below directly. On the other hand, thanks to the compactness property of the critical element and the conditional inequality (see the second item in Lemma \ref{lem:cond:2}) we can bound this variation by an arbitrarily negative number. Combining these two sides leads to the contradiction, unless the critical element vanishes identically (but this is not possible thanks to the initial assumption).
	\end{enumerate}

	We end this introduction part by briefly describing how the article is organized in the remaining part. In the first section \ref{sec:pre}, we recall some basic functional setting-ups, using which we obtain the local Cauchy theory for \eqref{eq:skg:int}; we also establish that the equation \eqref{eq:skg:int} is invariant under the Lorentz transform and how its associated energy functional varies under the Lorentz group. In the next three sections, we give our main technical tools: in Section \ref{sec:var}, we recall an alternative variational characterization of the infimum \eqref{eq:inf} and give some conditional inequality what plays an important role in the proof of our main result; in Section \ref{sec:lin:pro}, we prove the linear profile decomposition for any family of linear evolutions whose energy is uniformly bounded; in Section \ref{sec:pert}, we prove the perturbation lemma. In the last section \ref{sec:proof}, we give the proof of Theorem \ref{thm:main} following the outline sketched above.

\section{Preliminaries}\label{sec:pre}


\subsection{Function spaces and Strichartz estimates}
Let us first introduce the Littlewood-Paley decomposition.
Take $\varphi$ to be a smooth bump function satisfying $\varphi(\xi)=1$ for $|\xi|<1$ and $\varphi(\xi)=0$ for $|\xi|>2$. For each integer $j$, we set $\psi_j(\xi)=\varphi(2^{-j}\xi)-\varphi(2^{-j+1}\xi)$. Then we have
\begin{equation*}
	1=\sum_{j\in\mathbb{Z}}\psi_j(\xi),\ \ \ \forall \xi\neq 0
\end{equation*}	
and
\begin{equation*}
	1=\varphi(\xi) +\sum_{j>0}\psi_j(\xi),\ \ \ \forall \xi.
\end{equation*}

We next associate to each piece of the partition of unity an operator:
\begin{align*}
	P_jf &:= \left(\psi_j\hat{f}\right)^{\vee},\ \ \forall j>0\\
	P_0f &:= \left(\varphi\hat{f}\right)^{\vee}.
\end{align*}
For any given $p\geq 1$ and any real number $\sigma$, we define the Besov space $B^{\sigma}_{p,2}$ to consist of functions $f$ satisfying
\begin{equation*}
	\left\|P_0f\right\|_{L^p}+\left(\sum_{j\geq 1}2^{2\sigma j}\left\|P_jf\right\|^2_{L^p}\right)^{\frac{1}{2}}<\infty.
\end{equation*}
For such an element $f$ in $B^{\sigma}_{p,2}$, we will denote by $\left\|f\right\|_{B^{\sigma}_{p,2}}$ the finite quantity on the left hand side of the above inequality.

The relationship between Besov spaces and Sobolev spaces is summarized as follows.
\begin{proposition}[\cite{nakanishischlag2011}]
	On $\mathbb{R}^3$, there hold the following inequalities.
	\begin{itemize}
		\item For each $2\leq p<\infty$
		\begin{equation}
			\left\|f\right\|_{L^p}\leq C \left\|f\right\|_{B^0_{p,2}}.
		\end{equation}
		where $C$ is a constant.
		\item For any $2\leq q\leq p<\infty$ satisfying $\frac{1}{q}-\frac{1}{p}=\frac{\sigma}{3}$ with $\sigma\geq 0$,
		\begin{equation}
			B^{\sigma}_{q,2}\hookrightarrow B^0_{p,2}\hookrightarrow L^p.
		\end{equation}
		In particular
		\begin{equation}
			B^{\frac{1}{3}}_{\frac{18}{5},2}\hookrightarrow L^6,\ \ B^{\frac{1}{6}}_{6,2}\hookrightarrow L^9.
		\end{equation}
	\end{itemize}
\end{proposition}

{Here we recall a Bernstein inequality, that will be used later on.

Throughout, we shall call a ball  any set $\{\xi\in  \mathbb{R}^3 :~ |\xi| \leq  R\}$ with $R > 0$ and
an annulus any set $\{\xi \in \mathbb{R}^3 : ~0 < r_1 \leq  |\xi| \leq r_2  \}$ with 0$ < r_1 < r_2$ .
\begin{lemma}\label{bernstein}(see \cite{BCD11})
 Let  $\mathcal{B}$ be a ball and $\mathcal{C}$ an annulus. A constant $C$ exists such that
for any nonnegative integer $\alpha$, any couple $(p, q)$ in $[1, \infty]^2$ with $q \geq p \geq 1$, and
any function $u\in L^p(\mathbb{R}^3)$ , we have
	\begin{equation}\label{ball}
		\text{Supp}~ \hat{u}~\subset 2^l B \Rightarrow \|D^j u\|_{L^q(\mathbb{R}^3)}:=\sup_{j=\alpha}\|\partial^j u\|_{L^q(\mathbb{R}^3)}\leq C^{\alpha+1}2^{l\alpha+3l(\frac{1}{p}-\frac{1}{q})}\|u\|_{L^p(\mathbb{R}^3)},
	\end{equation}
	and
	\begin{equation}\label{annulus}
		\text{Supp}~ \hat{u}\subset~ 2^l \mathcal{C} \Rightarrow \|D^\alpha u\|_{L^q}:=\sup_{j=\alpha}\|\partial^j u\|_{L^q(\mathbb{R}^3)}\leq C^{\alpha+1}2^{l\alpha}\|u\|_{L^p(\mathbb{R}^3)}.
	\end{equation}
\end{lemma}
}
\begin{proposition}\label{prop:strichartz} The following assertions hold.
	\begin{itemize}
		\item The Strichartz estimate for free Klein-Gordon operator
		\begin{equation}
			\left\|e^{it\langle\nabla\rangle}f\right\|_{L^3_tB^{\frac{4}{9}}_{\frac{18}{5},2}(\mathbb{R}^{1+3}_{t,x})}\leq C\left\|f\right\|_{H^1(\mathbb{R}^3)}
		\end{equation}
		where the positive constant $C$ is universal.
		\item Any solution $u$ of the inhomogeneous equation
		\begin{equation}
			\Box u+u = F,\ u(0)=u_0,\ u_t(0)=u_1
		\end{equation}
		in $\mathbb{R}_t\times \mathbb{R}^3_x$ satisfies the estimates
		\begin{equation}\label{eq:122703}
			\left\|u\right\|_{L^3_tL^6_x(\mathbb{R}^{1+3}_{t,x})}\leq C\left(\|u_0\|_{H^1}+\|u_1\|_{L^2}+\|F\|_{L^1_tL^2_x}\right).
		\end{equation}
		and
		\begin{equation}
			\left\|u\right\|_{L^2_tB^{\frac{1}{6}}_{6,2}\cap L^\infty_tH^1_x}\leq C\left(\|u_0\|_{H^1}+\|u_1\|_{L^2}+\left\|F\right\|_{L^2_tB^{\frac{5}{6}}_{\frac{6}{5},2}+L^1_tL^2_x}\right)
		\end{equation}		
	\end{itemize}
\end{proposition}

\begin{remark}
	By Sobolev embedding $B^{\frac{1}{6}}_{6, 2}\hookrightarrow L^9$, we have the same bound for $L^2_tL^9_x$-norm of the solution. Interpolating with $L^\infty_tL^6$-norm bound, yields the $L^{\frac{8}{3}}_tL^8_x$-norm of the solution.
\end{remark}

\subsection{Basic Cauchy theory}
	Throughout the whole article, we use the following notations. We denote
	\begin{equation*}
		U:=\left(u_1,u_2\right)^{\mathsf{T}}\in H^1\times H^1
	\end{equation*}
	and  $$\vec{U}=\left((u_{11},u_{12}),(u_{21},u_{22})\right)^{\mathsf{T}}\in\left(H^1\times L^2\right)\times \left(H^1\times L^2\right)=:\mathcal{H}\times\mathcal{H}.$$ For such a data $\vec{U}$, we denote the free evolution as
	$$U(t)=\begin{pmatrix}u_1(t)\\ u_2(t)\end{pmatrix}=\cos\left(t\langle\nabla\rangle\right)\left(\begin{matrix}u_{11}\\ u_{21}\end{matrix}\right)+\frac{\sin\left(t\langle\nabla\rangle\right)}{\langle\nabla\rangle}\left(\begin{matrix}u_{12}\\ u_{22}\end{matrix}\right)=:S(t)(\vec{U})$$ and we also use the following notation of the quadruple
	\begin{equation}\label{eq:free:evo}
		\vec{S}(t)(\vec{U}):=\big((u_1(t),\partial_tu_1(t)),(u_2(t),\partial_tu_2(t))\big)^{\mathsf{T}}.
	\end{equation}
	
	\begin{definition}
		By a strong solution $U=(u_1,u_2)$ to \eqref{eq:skg:int} on the time interval $[0,T)$ for some $T>0$, we mean
		\begin{itemize}
			\item $U$ belongs to $\left(C\left([0,T),H^1\right)\cap C^1([0,T),L^2)\right)\times \left(C\left([0,T),H^1\right)\cap C^1([0,T),L^2)\right)$ and
			\item $U$ can be given as the vector version of Duhamel integration
			\begin{equation*}
				U=\cos\left(t\langle\nabla\rangle\right)\left(\begin{matrix}u_{11}\\ u_{21}\end{matrix}\right)+\frac{\sin\left(t\langle\nabla\rangle\right)}{\langle\nabla\rangle}\left(\begin{matrix}u_{12}\\ u_{22}\end{matrix}\right) +\int_0^t\frac{\sin\left((t-s)\langle\nabla\rangle\right)}{\langle\nabla\rangle}\left(\begin{matrix}u_1^3+\beta u_1u^2_2\\ u_2^3+\beta u_2u_1^2\end{matrix}\right)(s)ds.
			\end{equation*}
		\end{itemize}
	\end{definition}
	One immediate consequence of Minkowski's inequality is the energy estimate
	\begin{equation}\label{eq:energy}
		\left\|\vec{U}(t)\right\|_{\mathcal{H}\times\mathcal{H}} \lesssim \left\|\vec{U}\right\|_{\mathcal{H}\times\mathcal{H}}+\int_0^t\left\|\left(\begin{matrix}u_1^3+\beta u_1u^2_2\\ u_2^3+\beta u_2u_1^2\end{matrix}\right)(s)\right\|_{L^2\times L^2}ds,\ \ \forall t\geq 0.
	\end{equation}
	\begin{proposition}\label{prop:cauchy}
	For any $\vec{U}=\left((u_{11},u_{12}),(u_{21},u_{22})\right)^{\mathsf{T}}$ $\in\mathcal{H}^2$,
	the Cauchy problem \eqref{eq:skg:int} possesses a unique solution
	\begin{equation*}
		U(t)=\left(u_1, u_2\right)^{\mathsf{T}}\in \left(C([0,T),H^1)\cap C^1([0,T),L^2)\right)^2
	\end{equation*}
	for some $T\geq T_0>0$, where $T_{0}\sim \left\|\vec{U}(0)\right\|^{-2}_{\mathcal{H}\times\mathcal{H}}$.
	Furthermore,
	\begin{enumerate}[(i)]
		\item for any $t\in [0,T)$,  $\vec{U}(t)\in\mathcal{H}\times\mathcal{H}$, where $T$ is any positive number such that the solution exists on $[0,T)$.
		
		\item the energy functional
		\begin{equation*}
			E[\vec{U}(t)]:=\frac{1}{2}\int_{\mathbb{R}^3}\sum_{i=1}^2\left[|u_i|^2+|\nabla_{t,x} u_i|^2\right]-\frac{1}{4}\int_{\mathbb{R}^3}\left[u_1^4+u_2^4+2\beta u_1^2u^2_2\right]
		\end{equation*}
		is independent of $t$ for $t<T$.
		
		\item Let $T_\ast>0$ denote the maximal forward time of existence, then $T^\ast<\infty$ implies
		\begin{equation}\label{eq:strichartz:big}
			\left\|U\right\|_{\left(L^3([0,T_\ast),L^6(\mathbb{R}^3))\right)^2}=\infty.
		\end{equation}
		\item If $T^\ast=\infty$ {and $\left\|U\right\|_{\left(L^3_tL^6_x\right)^2}<\infty$}, then $\vec{U}$ scatters, that is, there exists $\vec{V}\in\mathcal{H}\times\mathcal{H}$ such that
		\begin{equation*}
			\left\|\vec{U}(t)-\vec{S}(t)\vec{V}\right\|_{\mathcal{H}\times\mathcal{H}}\rightarrow_{t\rightarrow\infty}0.
		\end{equation*}
		{Conversely, if $U$ scatters}, then one has
		\begin{equation}\label{eq:str:0313}
			\left\|U\right\|_{\left(L^3_t([0,\infty),L^6_x(\mathbb{R}^3))\right)^2}<\infty.
		\end{equation}
		\item if $\left\|\vec{U}(0)\right\|_{\mathcal{H}\times\mathcal{H}}\ll 1$, then the solution exist globally in time and
		\begin{equation}\label{eq:stri:smll}
			\left\|U\right\|_{\left(L^3([0,\infty),L^6(\mathbb{R}^3))\right)^2}\lesssim \left\|\vec{U}\right\|_{\mathcal{H}\times\mathcal{H}}.
		\end{equation}
		\item One also has the finite speed of propagation.
	\end{enumerate}
	\end{proposition}
	\begin{proof}
		Denote $X:=\left(L^\infty([0,T),H^1)\times L^\infty([0,T),L^2)\right)^2$ for some $T>0$. For any $t\in[0,T)$, we get from \eqref{eq:energy} by using the Sobolev imbedding $H^1\rightarrow L^6$
		\begin{equation}\label{eq:122701}
			\left\|\vec{U}(t)\right\|_{\mathcal{H}\times\mathcal{H}} \lesssim \left\|\vec{U}(0)\right\|_{\mathcal{H}\times\mathcal{H}}+T\left\|\vec{U}\right\|^3_X.
		\end{equation}
		This estimate together with a difference inequality allows us to use contraction principle to show local well-posedness of \eqref{eq:skg:int}, see \cite{Cui2022}. This shows $(i)$. For the second item $(ii)$, see also \cite{Cui2022}.
		
		For the third item $(iii)$, if $T_\ast<\infty$, then \eqref{eq:122701} implies
		\begin{equation}\label{eq:122702}
			\left\|\vec{U}\right\|_{L^\infty([0,T_\ast),\mathcal{H}^2)}=\infty.
		\end{equation}
		For otherwise, we can play the contraction mapping argument to show the local-in-time existence and hence extends the solution beyond $T_\ast$. Substituting \eqref{eq:122702} into \eqref{eq:energy}, yields \eqref{eq:strichartz:big}.
		
		In the situation of $(v)$, we can not use energy estimate to show global existence. However, the Strichartz estimate \eqref{eq:122703} allows to obtain for any $I=[0,T)$
		\begin{equation*}
			\left\|U\right\|_{L^3(I,L^6)\times L^3(I,L^6)}\lesssim \left\|\vec{U}(0)\right\|_{\mathcal{H}\times\mathcal{H}} + \left\|U\right\|_{L^3(I,L^6)\times L^3(I,L^6)}^3.
		\end{equation*}
		Based on this inequality and the smallness condition $\left\|\vec{U}(0)\right\|_{\mathcal{H}\times\mathcal{H}}\ll 1$, we can use a continuity argument to take $I=[0,\infty)$.
		
		For items $(iv)$ and $(vi)$, we can argument in the same lines as in \cite{nakanishischlag2011}.
	\end{proof}

\subsection{Lorentz symmetries}
 	Besides the invariances under translation in space and time, and the invariance under rotations in $\mathbb{R}^3$, the equation \eqref{eq:skg:int} is also invariant under the Lorentz group of the Minkowski space $\mathbb{R}\times\mathbb{R}^3$. For notational convenience, in this subsection we will denote the element $(t,x_1,x_2,x_3)$ in $\mathbb{R}\times\mathbb{R}^3$ by $(x_0,x_1,x_2,x_3)$.

 	The Lorentz group is generated by three families of coordinate exchanges $$L^\lambda_j(x_0,x_1,x_2,x_3)=:(y_0,y_1,y_2,y_3),j=1,2,3;\lambda\in\mathbb{R}$$
	where for each $j=1,2,3$ and $\lambda$, the transform $L^\lambda_j$ is defined as
 	\begin{equation*}
 		\left\{
 		\begin{split}
 			y_0&:= x_0\cosh\lambda + x_j\sinh\lambda;\\
 			y_j&:=x_0\sinh\lambda + x_j\cosh\lambda;\\
 			y_k&:=x_k,\ k\neq 0,j.
 		\end{split}
 		\right.
 	\end{equation*}

Each transform $L^\lambda_j$ induces the following transformation on the states (which we still denote by $L^\lambda_j$)
\begin{equation}
	L^\lambda_jU(x_0,x_1,x_2,x_3)=U(y_0,y_1,y_2,y_3).
\end{equation}
Define momentum
	\begin{equation}
		P_j\left[\overrightarrow{L^\lambda_jU}\right] := \left\langle \partial_tu^\lambda_1,\partial_{x_j}u^\lambda_1\right\rangle + \left\langle \partial_tu^\lambda_2,\partial_{x_j}u^\lambda_2\right\rangle,
	\end{equation}
we have Energy-Momentum relations
	\begin{equation}
 		\left\{
 		\begin{split}
 			 E\left[\overrightarrow{L^\lambda_jU}\right]&=  E\left[\overrightarrow{U}\right]\cosh\lambda + P_j\left[\overrightarrow{U}\right]\sinh\lambda;\\
 			P_j\left[\overrightarrow{L^\lambda_jU}\right]&=  E\left[\overrightarrow{U}\right]\sinh\lambda + P_j\left[\overrightarrow{U}\right]\cosh\lambda;\\
 		\end{split}
 		\right.
 	\end{equation}
This follows from  initial conditions
		$$E\left[\overrightarrow{L^\lambda_jU}\right]\Big |_{\lambda=0}=E\left[U\right], \partial_\lambda E\left[\overrightarrow{L^\lambda_jU}\right]\Big |_{\lambda=0}=P_j\left[U\right]$$
	  and
	\begin{proposition}\label{prop:lorentz}
		Let $U$ be a solution to \eqref{eq:skg:int}, then for each $j=1,2,3$ and $\lambda$
	\begin{itemize}
		\item $L^\lambda_jU$ is a solution as well;
		\item there holds the relation:
		\begin{align}\label{DEP}
			\partial_\lambda E\left[\overrightarrow{L^\lambda_jU}\right]&=P_j\left[\overrightarrow{L^\lambda_jU}\right]
		\end{align}
	and
	\begin{align}\label{DPE}
			\partial_\lambda P_j\left[\overrightarrow{L^\lambda_jU}\right]&=E\left[\overrightarrow{L^\lambda_jU}\right].
		\end{align}
	\end{itemize}
	\end{proposition}

	\begin{proof}
		The proof follows as a direct computation.
		Note $L^\lambda_jL^\beta_jU=L^{\lambda+\beta}_jU$ for $\lambda,\beta\in\mathbb{R}$. Using the infinitesimal identities $\partial_\lambda|_{\lambda=0}y_0=y_j$ and $\partial_\lambda|_{\lambda=0}y_j=y_0$, we compute
		\begin{equation}
			\partial_\lambda L^\lambda_jU=\partial_\beta|_{\beta=0}L^{\lambda+\beta}_jU=L^\lambda_j[(x_j\partial_t+t\partial_j)U].
		\end{equation}
		
	 For convenience, we denote $\partial_\beta^0:=\partial_\beta|_{\beta=0}$.  Denoting $U^\lambda:=L^\lambda_jU$, we first compute
		\begin{align*}
			\partial_\lambda E[\vec{U^\lambda}] & = \partial_\beta|_{\beta=0}E[L^\beta_j\vec{U}^\lambda]\\
			& = \left\langle \partial_tu^\lambda_1,\partial^0_\beta\partial_tL^\beta_ju^\lambda_1\right\rangle + \left\langle \nabla u^\lambda_1,\partial^0_\beta\nabla L^\lambda_ju^\lambda_1\right\rangle + \left\langle u^\lambda_1-\left(u^\lambda_1\right)^3-\beta\left(u^\lambda_2\right)^2u^\lambda_1,\partial^0_\beta L^\beta_ju^\lambda_1\right\rangle \\
			&\ \ \ \ \ + \left\langle \partial_tu^\lambda_2,\partial^0_\beta\partial_tL^\beta_ju^\lambda_2\right\rangle + \left\langle \nabla u^\lambda_2,\partial^0_\beta\nabla L^\lambda_ju^\lambda_2\right\rangle + \left\langle u^\lambda_2-\left(u^\lambda_2\right)^3-\beta\left(u^\lambda_1\right)^2u^\lambda_2,\partial^0_\beta L^\beta_ju^\lambda_2\right\rangle\\
			&= \left\langle \partial_tu^\lambda_1,x_j\partial^2_tu^\lambda_1+t\partial^{2}_{tx_j}u^\lambda_1+\partial_{x_j}u^\lambda_1\right\rangle + \left\langle \partial_{x_k} u^\lambda_1,x_j\partial_{tx_k}u^\lambda_1+t\partial_{x_kx_j}u^\lambda_1 + \delta_{jk}\partial_tu^\lambda_1\right\rangle\\
			&\ \ \ \ \ \ \ \ \ \ \ \ \ \ \ \ \ \ \ \ \ \ \ \ \ \  + \left\langle u^\lambda_1-\left(u^\lambda_1\right)^3-\beta\left(u^\lambda_2\right)^2u^\lambda_1,x_j\partial_tu^\lambda_1+t\partial_{x_j}u^\lambda_1\right\rangle\\
			&\ \ \ \ + \left\langle \partial_tu^\lambda_2,x_j\partial^2_tu^\lambda_2+t\partial^{2}_{tx_j}u^\lambda_2+\partial_{x_j}u^\lambda_2\right\rangle + \left\langle \partial_{x_k} u^\lambda_2,x_j\partial_{tx_k}u^\lambda_2+t\partial_{x_kx_j}u^\lambda_2 + \delta_{jk}\partial_tu^\lambda_2\right\rangle\\
			&\ \ \ \ \ \ \ \ \ \ \ \ \ \ \ \ \ \ \ \ \ \ \ \ \ \  + \left\langle u^\lambda_2-\left(u^\lambda_2\right)^3-\beta\left(u^\lambda_1\right)^2u^\lambda_2,x_j\partial_tu^\lambda_2+t\partial_{x_j}u^\lambda_2\right\rangle\\
			&= \left\langle x_j\partial_tu^\lambda_1,\partial^2_{tt}u^\lambda_1-\Delta u^\lambda_1+ u^\lambda_1 - \left(u^\lambda_1\right)^3-\beta\left(u^\lambda_2\right)^2u^\lambda_1\right\rangle + \left\langle \partial_tu^\lambda_1,\partial_{x_j}u^\lambda_1\right\rangle\\
			&\ \ \ \ + \left\langle x_j\partial_tu^\lambda_2,\partial^2_{tt}u^\lambda_2-\Delta u^\lambda_2+ u^\lambda_2 - \left(u^\lambda_2\right)^3-\beta\left(u^\lambda_1\right)^2u^\lambda_2\right\rangle + \left\langle \partial_tu^\lambda_2,\partial_{x_j}u^\lambda_2\right\rangle\\
			& =   \left\langle \partial_tu^\lambda_1,\partial_{x_j}u^\lambda_1\right\rangle + \left\langle \partial_tu^\lambda_2,\partial_{x_j}u^\lambda_2\right\rangle\equiv P_j[U^\lambda].
		\end{align*}
{	Next,	we turn to prove \eqref{DPE}. We compute
		\begin{align*}
			\partial_\lambda P_j[\vec{U^\lambda}] & = \partial_\beta|_{\beta=0}P_j[L^\beta_j\vec{U}^\lambda]\\
			&= \left\langle \partial_{x_j}u^\lambda_1,x_j\partial^2_tu_1^\lambda+t\partial^{2}_{tx_j}u^\lambda_1+\partial_{x_j}u^\lambda_1\right\rangle + \left\langle \partial_{t} u^\lambda_1,x_j\partial_{tx_j}u^\lambda_1+t\partial_{x_j}^2u^\lambda_1 + \partial_tu_1^\lambda\right\rangle\\
            &\ \ \ \ + \left\langle \partial_{x_j}u^\lambda_2,x_j\partial^2_tu_2^\lambda+t\partial^{2}_{tx_j}u^\lambda_2+\partial_{x_j}u^\lambda_2\right\rangle + \left\langle \partial_{t} u^\lambda_2,x_j\partial_{tx_j}u^\lambda_2+t\partial_{x_j}^2u^\lambda_2 + \partial_tu^\lambda_2\right\rangle\\
			&=  \underbrace{\left\langle \partial_{x_j}u^\lambda_1,x_j\partial^2_tu_1^\lambda  \right\rangle+\left\langle \partial_{x_j}u^\lambda_1,\partial_{x_j}u^\lambda_1 \right\rangle+\left\langle \partial_{x_j}u^\lambda_2,x_j\partial^2_tu_2^\lambda  \right\rangle+\left\langle \partial_{x_j}u^\lambda_2,\partial_{x_j}u^\lambda_2 \right\rangle}_{I_1} \\
&\ \ \ \   + \underbrace{\left\langle \partial_{x_j}u^\lambda_1,t\partial^{2}_{tx_j}u^\lambda_1  \right\rangle+\left\langle \partial_{t}u^\lambda_1,t\partial^{2}_{x_j}u^\lambda_1  \right\rangle + \left\langle \partial_{x_j}u^\lambda_2,t\partial^{2}_{tx_j}u^\lambda_2  \right\rangle+\left\langle \partial_{t}u^\lambda_2,t\partial^{2}_{x_j}u^\lambda_2  \right\rangle}_{I_2}  \\
 &\ \ \ \ + \underbrace{\left\langle \partial_{t}u^\lambda_1,x_j\partial_{tx_j}^2u^\lambda_1\right\rangle+\left\langle \partial_{t}u^\lambda_1,\partial_{t}u^\lambda_1\right\rangle+\left\langle \partial_{t}u^\lambda_2,x_j\partial_{tx_j}^2u^\lambda_2\right\rangle+\left\langle \partial_{t}u^\lambda_2,\partial_{t}u^\lambda_2\right\rangle}_{I_3} \\
		\end{align*}
By integration by parts, we have 	
	\begin{equation}\label{I12}
	  	I_2=0,~I_3=\frac{1}{2}\left(\left\langle \partial_{t}u_1^\lambda,\partial_{t}u_1^\lambda\right\rangle+\left\langle \partial_{t}u_2^\lambda,\partial_{t}u_2^\lambda\right\rangle\right).
	\end{equation}	
	Following from the fact that $$E[\vec{U}^\lambda]=J[U^\lambda]+\frac{1}{2}\|\partial_tU^\lambda\|_{L^2\times L^2}^2=J[U^\lambda]+I_3,$$
    it remains to check that 
    \begin{equation}\label{IJ}
        I_1=J[U^\lambda].
    \end{equation}
     Indeed, by using equations of $U^\lambda$,
 $I_1$ can be rewritten as
	\begin{equation}\label{I11}
		\left\langle x_j\partial_{x_j}u_1^\lambda,\Delta u_1^\lambda-u_1^\lambda-(u_1^\lambda)^3-\beta (u_2^\lambda)^2u_1^\lambda  \right\rangle+\left\langle x_j\partial_{x_j}u_2^\lambda,\Delta u_2^\lambda-u_2^\lambda-(u_2^\lambda)^3-\beta (u_1^\lambda)^2u_2^\lambda  \right\rangle
	\end{equation}
plus 
    \begin{equation}\label{I12}
		\left\langle \partial_{x_j}u^\lambda_1,\partial_{x_j}u^\lambda_1  \right\rangle+\left\langle \partial_{x_j}u^\lambda_2,\partial_{x_j}u^\lambda_2  \right\rangle. 
	\end{equation}
	By integration by parts, for $i=1,2$, $\left\langle x_j\partial_{x_j}u_i^\lambda,\Delta u_i^\lambda  \right\rangle$ equals to
	\begin{equation}\label{I111}
		-\left\langle x_j\partial_{x_j}\partial_{x_k}u_i^\lambda,\partial_{x_k} u_i^\lambda \right\rangle-\left\langle \delta_{jk}\partial_{x_j}u_i^\lambda,\partial_{x_k} u_i^\lambda \right\rangle=\frac{1}{2}\|u_i^\lambda\|_{\dot{H}^1}^2-\left\langle \partial_{x_j}u^\lambda_i,\partial_{x_j}u^\lambda_i  \right\rangle.
	\end{equation}
    and 
    \begin{align*}
        \left\langle x_j\partial_{x_j}u_1^\lambda,-u_1^\lambda-(u_1^\lambda)^3-\beta (u_2^\lambda)^2u_1^\lambda  \right\rangle &+\left\langle x_j\partial_{x_j}u_2^\lambda,-u_2^\lambda-(u_2^\lambda)^3-\beta (u_1^\lambda)^2u_2^\lambda  \right\rangle \\ &=\frac{1}{2}\|U^\lambda\|_{L^2\times L^2}^2-\frac{1}{4}\int_{\mathbb{R}^3}\left[u_1^4+u_2^4+2\beta u_1^2u^2_2\right]
    \end{align*}
	Combining this with \eqref{I12} and \eqref{I111}, we can get \eqref{IJ}. Then we finish the proof.}
\end{proof}

\section{Variational results}\label{sec:var}
	In this section, we recall some facts on variational characterization of ground state solution of \eqref{eq:skg:int} and prove some conditional inequalities that is helpful in the proof of our main result.
	
	We recall the following functional
	\begin{equation}
		\label{J} J[\varphi_1,\varphi_2]:=\frac{1}{2}\int_{\mathbb{R}^3}\sum_{j=1}^2\left[|\nabla\varphi_j|^2+|\varphi_j|^2\right]\mathrm{d}x-\frac{1}{4}\int_{\mathbb{R}^3}\left[\varphi_1^4+\varphi_2^4+2\beta\varphi_1^2\varphi_2^2\right]\mathrm{d}x
	\end{equation}
	and define the scaling of $(\varphi_1,\varphi_2)$ by ${(\varphi_1^\lambda,\varphi_2^\lambda)}:=e^\lambda(\varphi_1,\varphi_2)$, then
	\begin{equation}
		\label{K0} K_0[\varphi_1,\varphi_2]:=\frac{d}{d\lambda}\Big|_{\lambda=0}J[{\varphi_1^\lambda,\varphi_2^\lambda}].
	\end{equation}

	Observe that both $J$ and $K_0$ enjoy the mountain-pass property, see \cite{Cui2022}. Therefore, we can do some normalization to assume that
	\begin{equation}
		K_0[\varphi_1,\varphi_2]=0=\frac{d}{d\lambda}\Big|_{\lambda=0}J[\varphi_1^\lambda,\varphi_2^\lambda].
	\end{equation}
	
	The height of mountain pass over the ridge is
	\begin{equation}
		h_0:=\inf\left\{J[\varphi_1,\varphi_2]:(\varphi_1,\varphi_2)\in \left(H^1\times H^1\right)\backslash\{(0,0)\},\ K_0[\varphi_1,\varphi_2]=0\right\}.
	\end{equation}
	By introducing
	\begin{equation}\label{G0}
		G_0[\varphi_1,\varphi_2]:=J[\varphi_1,\varphi_2]-\frac{1}{4}K_0[\varphi_1,\varphi_2]=\frac{1}{4}\left\|(\varphi_1,\varphi_2)\right\|^2_{H^1\times H^1}
	\end{equation}
	we have
	\begin{lemma}[\cite{Cui2022,sirakov07}]\label{lem:h0}One has the following alternative characterization of $h_0$:
		\begin{equation}
			h_0=\inf\left\{G_0[\varphi_1,\varphi_2]:(\varphi_1,\varphi_2)\in\left(H^1\times H^1\right)\backslash\{(0,0)\},K_0[\varphi_1,\varphi_2]\leq0\right\}.
		\end{equation}
		What's more, $h_0$ is attained by some element $(Q_1,Q_2)\in\left(H^1\times H^1\right)\backslash\{(0,0)\}$.
	\end{lemma}

	With the notations as in this lemma, we also have the following conditional inequalities, which is of interest in itself.
	\begin{lemma}\label{lem:fcn:1}
		For $J$ and $K_0$, one has the following two conditional inequalities.
		\begin{enumerate}[(i)]
			\item For $(\varphi_1,\varphi_2)\in H^1\times H^1$, if $J[\varphi_1,\varphi_2]<J[Q_1,Q_2]$ and $K_0[\varphi_1,\varphi_2]<0$, there holds
			\begin{equation}\label{k0-}
				-K_0[\varphi_1,\varphi_2]\geq 2\left(J[Q_1,Q_2]-J[\varphi_1,\varphi_2]\right).
			\end{equation}
			\item For $(\varphi_1,\varphi_2)\in H^1\times H^1$, if $J[\varphi_1,\varphi_2]<J[Q_1,Q_2]$ and $K_0[\varphi_1,\varphi_2]\geq 0$, there holds
			\begin{equation}\label{k0+}
				K_0[\varphi_1,\varphi_2]\geq c_0\min\left(J[Q_1,Q_2]-J[\varphi_1,\varphi_2],\left\|(\varphi_1,\varphi_2)\right\|^2_{H^1\times H^1}\right)
			\end{equation}
			for some positive number $c_0$.
		\end{enumerate}
	\end{lemma}
	
	\begin{proof}
		 Let $j(\lambda)=J[\varphi_1^\lambda,\varphi_2^\lambda]$ and
		$$n(\lambda)=\frac{1}{4}\int_{\mathbb{R}^3}\left( (\varphi_1^\lambda)^4+2\beta (\varphi_1^\lambda \varphi_2^\lambda)^2+(\varphi_2^\lambda)^4\right)\mathrm{d}x,$$
		then we get
        \begin{equation}\label{jK0}
            j'(\lambda)=K_0[\varphi_1^\lambda,\varphi_2^\lambda],
        \end{equation}
        \eqref{J} and \eqref{K0} can be rewritten as
		\begin{equation}\label{jK1}
             j(0)=J[\varphi_1,\varphi_2], ~j'(0)=K_0[\varphi_1,\varphi_2].
		\end{equation}
		A direct computation yield a differential equation
		\begin{equation} \label{estimation_of_j_1}
			j'' = 4 j'-4j- n' ,
		\end{equation}
		and a inequality of $j$
        \begin{equation}\label{estimation_of_j_2}
			j' \leq 2j,
		\end{equation}
        for any $\lambda\in \mathbb{R}$.
		
		Now we consider the case (i): $K_0[\varphi_1,\varphi_2]<0$. Notice that $n'=4n\geq 0$, then together with \eqref{estimation_of_j_1} and \eqref{estimation_of_j_2},
        \begin{equation}\label{j2}
            j''\leq 2 j',
        \end{equation}
        for any $\lambda\in \mathbb{R}$.

Recalling a Mountain Pass lemma in \cite{Cui2022}, we have that:
$$
  there~ exists~ a~ \lambda_0<0, ~such~ that~ j'(\lambda)<0~ for~ \lambda_0<\lambda\leq 0~ and ~j' (\lambda_0)=0.
$$  		
		Thus integrating \eqref{j2}  over $(\lambda_0,0)$, we obtain
		\begin{equation}
			\int_{\lambda_0}^0 j'' (\lambda) \mathrm{d}\lambda \leq 2 \int_{\lambda_0}^0 j' (\lambda) \mathrm{d}\lambda,
		\end{equation}
		 which is equivalent to
         \begin{equation}\label{case1}
            K_0[\varphi_1,\varphi_2]=j' (0)\leq 2 (j(0)-j(\lambda_0)).
		 \end{equation}
		Since $K_0[\varphi_1^{\lambda_0},\varphi_2^{\lambda_0}]=j' (\lambda_0)=0$ and $(\varphi_1^{\lambda_0},\varphi_2^{\lambda_0})\neq (0,0)$,
by using Lemma \ref{lem:h0}, we get
        \begin{equation}
            j(\lambda_0)=J[\varphi_1^{\lambda_0},\varphi_2^{\lambda_0}]\geq h_0=J[Q_1,Q_2].
        \end{equation}
        Combing this with \eqref{case1}, we get the desired estimation \eqref{k0-}.
				
		Next we consider the case (ii): $K_0[\varphi_1,\varphi_2]\geq 0$. By using \eqref{G0}, we have
        \begin{equation}\label{jKH1}
            j(0)=J[\varphi_1,\varphi_2]\geq \frac{1}{4}\|(\varphi_1,\varphi_2)\|_{H^1\times H^1}^2.
        \end{equation}

        Now we divided the proof of this part into two subcases.
		
		Subcase one. Suppose that  the following inequality is valid, for some $0<\epsilon<1$,
		\begin{equation}\label{assumption}
			6K_0[\varphi_1,\varphi_2]\geq 4J[\varphi_1,\varphi_2]+\epsilon (K_0[\varphi_1,\varphi_2]-\|(\varphi_1,\varphi_2)\|_{H^1\times H^1}^2).
		\end{equation}
		Then combing with \eqref{jKH1}, we can imply that
		\begin{equation}
			K_0[\varphi_1,\varphi_2]\geq \frac{1-\epsilon}{6-\epsilon} \|(\varphi_1,\varphi_2)\|_{H^1\times H^1}^2.
		\end{equation}
		
		Subcase two. If \eqref{assumption}  fails, then based on the definitions of $j$ and $n$, we can rewrite \eqref{assumption} as
		\begin{equation}\label{estimation_of_j_3}
			6 j'(0) <4 j(0)-\epsilon n'(0).
		\end{equation}	
		Together with \eqref{estimation_of_j_1} and $n'\geq 0$, we have, at $\lambda=0$,
		\begin{equation}\label{estimation_of_j_4}
			j'' <-2 j'.
		\end{equation}
        To obtain the \eqref{k0+}, it is suffices to prove the following claim:
        \begin{claim}\label{claim:j4}
			there exists $\lambda_1>0$, such that  \eqref{estimation_of_j_4} is valid for  any $0\leq \lambda\leq \lambda_1$ and $j'( \lambda_1)=0$.
		\end{claim}
        Indeed, if \eqref{estimation_of_j_4} is valid for $\lambda\in (0,\lambda_0)$, then integrating \eqref{estimation_of_j_4} over $(0,\lambda_0)$, one can obtain
		\begin{equation}
			-K_0[\varphi_1,\varphi_2]\leq -2(J[Q_1,Q_2]-J[\varphi_1,\varphi_2]).
		\end{equation}
		Let $c_0=\min\{\frac{1-\epsilon}{6-\epsilon},2\}$, we finish the proof of case (ii).


		Finally, it is turn to give the proof of  claim \ref{claim:j4}. It suffices to prove that \eqref{estimation_of_j_3} is valid for  any $0\leq \lambda\leq \lambda_1$ and
    \begin{equation}\label{j'0}
        j'( \lambda_1)=0.
    \end{equation}
 The existence of $\lambda_1$ in \eqref{j'0} comes from a Mountain Pass property in \cite{Cui2022}. More precisely, we have
     \begin{equation}
     there~ exists~ \lambda_1>0~ such~ that~
		K_0[\varphi_1^{\lambda_1},\varphi_2^{\lambda_1}]=0~ and~ for~ any~ \lambda<\lambda_1, K_0[\varphi_1^{\lambda},\varphi_2^{\lambda}]>0.
     \end{equation}
So when $\lambda$ increases, $j$ increases and as long as  \eqref{estimation_of_j_3} is valid, \eqref{estimation_of_j_4} is valid, so $j''$ is negative and $j'$ decreases. Another observation is that
		\begin{equation}
			n''=4n'= 16 n > 0.
		\end{equation}
		This implies that $n'$ increases. Therefore \eqref{estimation_of_j_3} is true until $\lambda=\lambda_1$. Thus we complete the proof of claim.
	\end{proof}

	Let us introduce another pair of functionals
	\begin{equation}
		K_2[\varphi_1,\varphi_2]:=\frac{d}{d\lambda}\Big|_{\lambda=0}J\left[e^{{\frac{3}{2}}\lambda}\varphi_1(e^\lambda\cdot),e^{{\frac{3}{2}}\lambda}\varphi_2(e^\lambda\cdot)\right]=\int_{\mathbb{R}^3}\left[\sum_{j=1}^2|\nabla\varphi_j|^2-\frac{3}{4}\left(\varphi_1^4+\varphi^4_2+2\beta\varphi_1\varphi_2\right)\right]dx
	\end{equation}
	and
	\begin{equation}
		G_2[\varphi_1,\varphi_2]:=J[\varphi_1,\varphi_2]-\frac{1}{3}K_2[\varphi_1,\varphi_2]=\frac{1}{6}\left(\left\|\nabla\varphi_1\right\|^2_{L^2}+\left\|\nabla\varphi_2\right\|^2_{L^2}\right)+\frac{1}{2}\left\|(\varphi_1,\varphi_2)\right\|^2_{L^2\times L^2},
	\end{equation}	
	We can also characterize $h_0$ variationally in terms of $G_2$ and $K_2$.
	\begin{lemma}[\cite{Cui2022}] Let $K_2$ and $G_2$ be defined as above, then one has
		\begin{equation}
			h_0=\inf\left\{G_2[\varphi_1,\varphi_2]:K_2[\varphi_1,\varphi_2]\leq 0,\ (\varphi_1,\varphi_2)\in\left(H^1\times H^1\right)\backslash\{(0,0)\}\right\}.
		\end{equation}
	\end{lemma}
	
	As in Lemma \ref{lem:fcn:1}, we also have the following conditional inequalities, which play an important role in proving our main result: Theorem \ref{thm:main}.
	\begin{lemma}\label{lem:cond:2}
		About the functionals $J$ and $K_2$, there hold the following two inequalities.
		\begin{enumerate}[(i)]
			\item For $(\varphi_1,\varphi_2)\in H^1\times H^1$, if $J[\varphi_1,\varphi_2]<J[Q_1,Q_2]$ and $K_2[\varphi_1,\varphi_2]<0$, one has
			\begin{equation}
				-K_2[\varphi_1,\varphi_2]\geq 2\left(J[Q_1,Q_2]-J[\varphi_1,\varphi_2]\right).
			\end{equation}
			\item For $(\varphi_1,\varphi_2)\in H^1\times H^1$, if $J[\varphi_1,\varphi_2]<J[Q_1,Q_2]$ and $K_2[\varphi_1,\varphi_2]\geq0$, one has
			\begin{equation}
				K_2[\varphi_1,\varphi_2]\geq c_1\min\left(J[Q_1,Q_2]-J[\varphi_1,\varphi_2],\left\|(\varphi_1,\varphi_2)\right\|^2_{H^1\times H^1}\right).
			\end{equation}
            for some constant $c_1>0$.
		\end{enumerate}
	\end{lemma}
	The proof precedes in the same spirit as in Lemma \ref{lem:fcn:1}, so it is ommitted here.

\section{Linear profile decomposition}\label{sec:lin:pro}
	As is alluded to in the introduction, the linear profile decomposition plays an important role in the proof of Theorem \ref{thm:main}. In this section, we give its statement together with its proof, which are adapted to the system under consideration from the scalar version \cite{nakanishischlag2011}. The decomposition of this kind has been attracted much attention, see for instance \cite{Gerard1998,Bahouri1999}. 
	
	\begin{proposition}\label{prop:lin:prof}
		Let $\left\{U_n\right\}$ be a sequence of functions obeying the free system
		\begin{equation}\label{eq:skg:noforce}
			\left\{
			\begin{split}
				\partial^2_{tt}u_1-\Delta u_1 + u_1 &=0, \mbox { in } \mathbb{R}\times\mathbb{R}^3,\\
				\partial^2_{tt}u_2-\Delta u_2 + u_2 &=0, \mbox { in } \mathbb{R}\times\mathbb{R}^3,\\
			\end{split}
			\right.
		\end{equation}	
		 which satisfies for each $n$
		\begin{equation}
			\left\|\vec{U}_n\right\|_{\left(L^\infty_t(\mathbb{R};\mathcal{H})\right)^2}\leq C<\infty
		\end{equation}
		for some positive constant $C$, independent of $n$. Then there exist a sub-sequence, {still denoted by $\left\{\vec{U}_n\right\}$,} a sequence  $\{\vec{V}^{(j)}\}$ of solutions to \eqref{eq:skg:noforce} that is bounded in $\mathcal{H}\times\mathcal{H}$, and a sequence $\left(t_n^{(j)},x_n^{(j)}\right)\in\mathbb{R}\times\mathbb{R}^3$, such that for all $\vec{\gamma}^{(k)}_n$ defined by
		\begin{equation}\label{eq:03153}
			\vec{\gamma}^{(k)}_n:=\vec{U}_n(t,x)-\sum_{0\leq j{\color{red}<} k}\vec{V}^{(j)}(t+t_n^{(j)},x+x_n^{(j)})
		\end{equation}
		there hold the following assertions:
		\begin{enumerate}[(i)]
			\item for any $0\leq j<k$
			\begin{equation}\label{eq:03154}
				\vec{\gamma}^{(k)}_n(\cdot -t_n^{(j)},\cdot-x^{(j)}_n)\rightharpoonup_{n\rightarrow\infty}0\ \ \mathrm{in}\ \ \mathcal{H}\times\mathcal{H}
			\end{equation}
			and
			\begin{equation}\label{eq:03155}
				\lim_{n\rightarrow\infty}\left(\left|t^{(j)}_n-t^{(k)}_n\right|+|x_n^{(j)}-x_n^{(k)}|\right)=\infty
			\end{equation}
			\item the sequence $\{\vec{\gamma}^{(k)}_n\}$ is small in the sense
			\begin{equation}
				\lim_{k\rightarrow\infty}\limsup_{n\rightarrow\infty}\left\|\gamma^{(k)}_n\right\|_{\left(L^\infty_tL_x^p\cap L^3_tL^6_x\right)^2}=0,\ \ \ \forall p\in(2,6)
			\end{equation}
			and for each fixed integer $k$, the energy decomposes asymptotically
			\begin{equation}\label{eq:03156}
				\left\|\vec{U}_n\right\|^2_{\mathcal{H}\times\mathcal{H}}=\sum_{0\leq j<k}\left\|\vec{V}^{(j)}\right\|^2_{\mathcal{H}\times\mathcal{H}}+\left\|\vec{\gamma}^{(k)}_n\right\|^2_{\mathcal{H}\times\mathcal{H}}+\mathbf{o}(1)
			\end{equation}
			as $n$ tends to infinity.
		\end{enumerate}
	\end{proposition}
	\begin{remark}
		The proof of this proposition is to `track' some proper norm of $\vec{\gamma}^{(k)}_n$. Thanks to the sub-critical nature of this norm, the proof is much simpler than that in \cite{Bahouri1999}.
	\end{remark}
	\begin{proof} The proof we are going to give is a combination of \cite{nakanishischlag2011} and \cite{Gerard1998}. We note that
			it suffices to show
			\begin{equation}\label{eq:lin:220}
				\lim_{k\rightarrow\infty}\limsup_{n\rightarrow\infty}\left\|\gamma^{(k)}_n\right\|_{\left(L^\infty_t(\mathbb{R},B^{-3/2}_{\infty,\infty})\right)^2}=0
			\end{equation}
			
		Indeed this follows from the interpolating argument.
			By the energy estimate, we get
			\begin{equation}\label{eq:lin:221}
				\sup_{k}\limsup_{n\rightarrow\infty}\left\|\gamma^{(k)}_n\right\|_{\left(L^\infty_t(\mathbb{R},B^{1}_{2,2}(\mathbb{R}^3))\right)^2}<\infty.
			\end{equation}
			Interpolating \eqref{eq:lin:220} and \eqref{eq:lin:221} yields
			\begin{equation}\label{eq:lin:222}					\sup_{k}\limsup_{n\rightarrow\infty}\left\|\gamma^{(k)}_n\right\|_{\left(L^\infty_t(\mathbb{R},B^{1/6}_{3,3}(\mathbb{R}^3))\right)^2}=0
			\end{equation}
			which in turn implies by using Sobolev embedding ({\cite{Adams-Fournier03}})
			\begin{equation}\label{eq:lin:223}
									\sup_{k}\limsup_{n\rightarrow\infty}\left\|\gamma^{(k)}_n\right\|_{\left(L^\infty_t(\mathbb{R},L^3(\mathbb{R}^3))\right)^2}=0
			\end{equation}
			For $p\in(2,3)$, we can interpolate this last estimate with the $L^2\times L^2$ bound to obtain desired result, while for $p\in(3,6)$, we can interpolate \eqref{eq:lin:223} with the $\dot{H}^1\times \dot{H}^1$ bound to achieve the results. {For the $\left(L^3_t(\mathbb{R},L^6_x(\mathbb{R}^3))\right)^2$ bound, we shall first use the Sobolev inequality in the time variable $t$ to gain $\frac{1}{3}$ derivative, then use the fact that one derivative in time is equivalent to one in the space variable, and finally use the Sobolev inequality in the space variable to reduce to the case just proved}.
		
		We are now going to show \eqref{eq:lin:220} by an inductive argument. For the notational convenience, we denote
		\begin{equation*}
			\vec{\gamma}^{(0)}_n:=\vec{U}_n
		\end{equation*}
		and set
	\begin{equation*}
		\nu^{(0)}:=\limsup_{n\rightarrow\infty}\left\|\gamma^{(0)}_n\right\|_{\left(L^\infty_t(\mathbb{R},B^{-3/2}_{\infty,\infty})\right)^2}.
	\end{equation*}
	Then we are facing two scenarios. In the first one $\nu^{(0)}=0$, we cease the induction by setting
	\begin{equation}
		\vec{V}^{(j)}=\left((0,0),(0,0)\right),\ \ \ \forall\ j\geq 1.
	\end{equation}
	In the second scenario $\nu^{(0)}>0$, we pick $\{k_n\}^\infty_{n=1}$ so that
	\begin{equation}
		2^{-\frac{3}{2}k_n}\left\|P_{k_n}\gamma^{(0)}_n\right\|_{L^\infty_{t,x}(\mathbb{R}\times\mathbb{R}^3)^2}>\frac{\nu^{(0)}}{2}
	\end{equation}
	Since $\left\|\vec{\gamma}^{(0)}_n\right\|_{\mathcal{H}\times\mathcal{H}}$ is bounded uniformly in $n$ and $t$, it follows that $\{k_n\}$ is indeed a bounded sequence. Thus we can assume that, up to subsequence there holds $k_n=k^{(0)}_{\infty}$, and for some sequence
	\begin{equation*}
		\left\{(t_n^{(0)},x_n^{(0)})\in\mathbb{R}\times\mathbb{R}^3\right\}
	\end{equation*}
	we have
	\begin{equation}
		2^{-\frac{3}{2}k^{(0)}_{\infty}}\left|P_{k^{(0)}_{\infty}}U_n(-t_n^{(0)},-x^{(0)}_n)\right|>\frac{\nu^{(0)}}{2},\ \ \forall n.
	\end{equation}
	Thanks to the uniform boundedness of $\left\|\vec{\gamma}^{(0)}_n\right\|_{\mathcal{H}\times\mathcal{H}}$, we assume that
	\begin{equation}\label{eq:gamma:2}
		\vec{\gamma}^{(0)}_n(-t_n^{(0)},\cdot-x^{(0)}_n)\rightharpoonup_{n\rightarrow\infty}\vec{V}^{(0)}(\cdot)=:\vec{V}^{(0)}(0,\cdot)\ \mathrm{in}\ \mathcal{H}\times\mathcal{H}
	\end{equation}
	for some $\vec{V}^{(0)}$ in $\mathcal{H}\times\mathcal{H}$. Denoting $\vec{V}^{(0)}(t,x)$ be the solution to \eqref{eq:skg:noforce} corresponding to data $\vec{V}^{(0)}$, we set for each $n$
	\begin{equation}
		\vec{\gamma}^{(1)}_n(t,x):=\vec{\gamma}^{(0)}_n(t,x)-\vec{V}^{(0)}\left(t+t^{(0)}_n,x+x^{(0)}_n\right)
	\end{equation}
	 Applying  Bernstein-type inequality \eqref{annulus}  in Lemma \ref{bernstein} with $p=2,q=\infty$, we pick a constant $C_0>0$ such that
	\begin{equation}\label{eq:constant:c0}
		C_0\left\|V^{(0)}(0)\right\|_{L^2\times L^2}\geq 2^{-\frac{3}{2}k^{(0)}_\infty}\left|P_{k^{(0)}_\infty}V^{(0)}(0,\cdot)\right|>\frac{\nu^{(0)}}{4}.
	\end{equation}

	Assume now for an integer $k$, we have the finite sequence $\{\nu^{(0)},\nu^{(1)},\nu^{(2)},\dots,\nu^{(k)}\}$ of strictly positive numbers, the finite sequence of couples of functions $\{\vec{V}^{(0)},\vec{V}^{(1)},\vec{V}^{(2)},\dots, \vec{V}^{(k)}\}$ that are solutions to \eqref{eq:skg:noforce} corresponding to the initial datum $\{\vec{V}^{(0)}(0),\vec{V}^{(1)}(0), \vec{V}^{(2)}(0),\dots, \vec{V}^{(k)}(0)\}$, a finite sequence of finite numbers $\{k^{(0)}_\infty,k^{(1)}_\infty, k^{(2)}_\infty, \dots, k^{(k)}_{\infty}\}$, and a finite sequence of sequences of space-time points $\{(t^{(0)}_n,x^{(0)}_n)_{n\in\mathbb{N}}, (t^{(1)}_n,x^{(1)}_n)_{n\in\mathbb{N}}, (t^{(2)}_n,x^{(2)}_n)_{n\in\mathbb{N}},\dots, (t^{(k)}_n,x^{(k)}_n)_{n\in\mathbb{N}}\}$ such that
	\begin{enumerate}[(i)]
		\item for each $l\in\{1,2,\dots, k\}$
		\begin{equation*}
			\vec{\gamma}^{(l)}_n(t,x):=\vec{\gamma}^{(l-1)}_n(t,x)-\vec{V}^{(l-1)}(t+t^{(l-1)}_n,x+x^{(l-1)}_n)
		\end{equation*}
			that converges weakly to zero in  $\mathcal{H}\times \mathcal{H}$.
		\item for each $l\in\{1,2,\dots, k\}$
		\begin{equation}
			2^{-\frac{3}{2}k^{(l)}_{\infty}}\left|P_{k^{(l)}_{\infty}}\gamma^{(l)}_n(-t_n^{(l)},-x^{(l)}_n)\right|>\frac{\nu^{(l)}}{2}.
		\end{equation}
		\item for the same constant $C_0$ as in \eqref{eq:constant:c0}, we have for each $l\in\{1,2,\dots, k\}$
		\begin{equation}
			C_0\left\|V^{(l)}(0)\right\|_{L^2\times L^2}\geq 2^{-\frac{3}{2}k^{(l)}_\infty}\left|P_{k^{(l)}_\infty}V^{(l)}(0,\cdot)\right|>\frac{\nu^{(l)}}{4}.
		\end{equation}
	\end{enumerate}

	We next implement the construction in the $(k+1)$-th step. We define
	\begin{equation}\label{eq:12192}
		\vec{\gamma}^{(k+1)}_n(t,x):=\vec{U}_n(t,x)-\sum_{0\leq j\leq k}\vec{V}^{(j)}(t+t^{(j)}_n,x+x^{(j)}_n)=\vec{\gamma}^{(k)}_n(t,x)-\vec{V}^{(k)}(t+t^{(k)}_n,x+x^{(k)}_n)
	\end{equation}
	and set
	\begin{equation}
		\nu^{(k+1)}:=\limsup_{n\rightarrow\infty}\left\|\gamma^{(k+1)}_n\right\|_{\left(L^\infty_t(\mathbb{R},B^{-3/2}_{\infty,\infty})\right)^2}.
	\end{equation}
	As in the {initial step} $k=0$, we are facing two scenarios as well. In the first scenario $\nu^{(k+1)}=0$, we terminate the induction by setting
	\begin{equation}
		\vec{V}^{(j)}=\left((0,0),(0,0)\right),\ \ \forall j\geq k+1.
	\end{equation}
	In the second scenario $\nu^{(k+1)}>0$, we can pick an integer $k^{(k+1)}_{\infty}\geq0$ and a sequence $\left(t^{(k+1)}_n,x^{(k+1)}_n\right)\in\mathbb{R}\times \mathbb{R}^3$ such that
	\begin{enumerate}[(a)]
		\item for all sufficiently large $n$, we have
		\begin{equation}
				2^{-\frac{3}{2}k^{(k+1)}_{\infty}}\left|P_{k^{(k+1)}_{\infty}}\gamma^{(k+1)}_n(-t_n^{(k+1)},-x^{(k+1)}_n)\right|>\frac{\nu^{(k+1)}}{2};
		\end{equation}
		\item for some $\vec{V}^{(k+1)}\in\mathcal{H}\times\mathcal{H}$, we have the weak convergence (up to subsequence)
		\begin{equation}
			\vec{\gamma}^{(k+1)}_n(-t_n^{(k+1)},\cdot-x^{(k+1)}_n)\rightharpoonup_{n\rightarrow\infty}\vec{V}^{(k+1)}(\cdot)=:\vec{V}^{(k+1)}(0,\cdot)\ \mathrm{in}\ \mathcal{H}\times\mathcal{H};
		\end{equation}
		\item for the same constant $C_0$, we have
		\begin{equation}\label{eq:03157}
				C_0\left\|V^{(k+1)}(0)\right\|_{L^2\times L^2}\geq 2^{-\frac{3}{2}k^{(k+1)}_\infty}\left|P_{k^{(k+1)}_\infty}V^{(k+1)}(0,\cdot)\right|>\frac{\nu^{(k+1)}}{4}.
		\end{equation}
	\end{enumerate}
	
	By the induction principle, we can finish the construction for each integer $k$. {Thus it remains to }show the asserted results.
	We first prove \eqref{eq:03155} by an inductive argument. As the first step, we show it holds with $j=0$ and $k=1$, that is, there holds
		\begin{equation}\label{eq:1206}
			\left|t^{(0)}_n-t^{(1)}_n\right|+\left|x^{(0)}_n-x^{(1)}_n\right|\rightarrow_{n\rightarrow\infty}\infty.
		\end{equation}
		
		We prove this assertion by contradiction, assuming
		\begin{equation}\label{eq:03150}
			t^{(0)}_n-t^{(1)}_n\rightarrow\tau\ \ \mathrm{and}\  \ x^{(0)}_n-x^{(1)}_n\rightarrow\xi
		\end{equation}
		for some $\tau\in\mathbb{R}$ and $\xi\in\mathbb{R}^3$ up to subsequences. By linearity, we rewrite
		\begin{equation}\label{eq:0315}
			\vec{\gamma}^{(1)}_n(-t^{(1)}_n,\cdot-x^{(1)}_n)=\vec{S}(t^{(0)}_n-t^{(1)}_n)\vec{\gamma}^{(1)}_n(-t^{(0)}_n,\cdot-x^{(0)}_n+x^{(0)}_n-x^{(1)}_n).
		\end{equation}
	Considering the second limitation in \eqref{eq:03150}, we infer from the weak limit \eqref{eq:gamma:2} that
	\begin{equation}\label{eq:12061}
		\vec{\gamma}^{(1)}_n(-t^{(0)}_n,\cdot-x^{(0)}_n+x^{(0)}_n-x^{(1)}_n)\rightharpoonup_{n\rightarrow\infty}0\ \ \mathrm{in}\ \ \mathcal{H}\times\mathcal{H}.
	\end{equation}
	Thus by \eqref{eq:0315} and \eqref{eq:12061}, to achieve the absurd (that the unique weak limit $\vec{V}^{(1)}$ is a non-zero element), it suffices to prove
	\begin{equation}\label{eq:12060}
		\vec{S}(t_n)\vec{\gamma}_n\rightharpoonup_{n\rightarrow\infty}0\ \ \mathrm{in}\ \ \mathcal{H}\times\mathcal{H}
	\end{equation}
	provided that
	\begin{equation}\label{eq:12062}
		t_n\rightarrow t\in\mathbb{R} \ \ \mathrm{and}\ \ \vec{\gamma}_n\rightharpoonup_{n\rightarrow\infty}0\ \ \mathrm{in}\ \ \mathcal{H}\times\mathcal{H}.
	\end{equation}
	We show this by a duality argument. Let $\vec{\psi}\in\mathcal{H}\times\mathcal{H}$ be an arbitrarily given quadruple of bump functions. To show \eqref{eq:12060} is equivalent to show
	\begin{equation}\label{eq:03152}
		\left\langle\vec{S}(t_n)\vec{\gamma}_n,\vec{\psi} \right\rangle_{(\mathcal{H}\times \mathcal{H})\times (\mathcal{H}\times \mathcal{H})}\rightarrow_{n\rightarrow\infty}0.
	\end{equation}
	For this, we use the unitaryness of $\vec{S}$ on $\mathcal{H}^2$ to rewrite
	\begin{align}\label{eq:03151}
		\left\langle\vec{S}(t_n)\vec{\gamma}_n,\vec{\psi} \right\rangle_{\mathcal{H}^2\times \mathcal{H}^2} &= \left\langle\vec{\gamma}_n,\vec{S}(t_n)\vec{\psi} \right\rangle_{\mathcal{H}^2\times \mathcal{H}^2}\nonumber\\ &= \left\langle\vec{\gamma}_n,\vec{S}(t)\vec{\psi} \right\rangle_{\mathcal{H}^2\times \mathcal{H}^2}+\left\langle\vec{\gamma}_n,\left[\vec{S}(t_n)-\vec{S}(t)\right]\vec{\psi} \right\rangle_{\mathcal{H}^2\times \mathcal{H}^2}.
	\end{align}
	Since $\vec{\gamma}_n$ converges weakly to $0$ in $\mathcal{H}\times\mathcal{H}$ and the linear evolution $\vec{S}(t)\vec{\psi}$ is still a test functions, the first term tends to $0$. By the strong continuity property of $\vec{S}(t)$ and the uniform boundedness of $\{\vec{\gamma}_n\}$, we see that the second term tends to $0$ as $n$ tends to infinity, thanks to the first limitation in \eqref{eq:12062} . Combining these analysis with \eqref{eq:03151} proves \eqref{eq:03152} and hence finishes the proof of \eqref{eq:1206}.	
	
	Assume now the asserted results \eqref{eq:03153} and \eqref{eq:03154} hold up to an integer $k$. Now we are going to show they hold as well up to $k+1$. We argue by contradiction, assuming for some $0\leq j\leq k$ there hold
	\begin{equation}\label{eq:12190}
			t^{(j)}_n-t^{(k+1)}_n\rightarrow\tau\ \ \mathrm{and}\  \ x^{(j)}_n-x^{(k+1)}_n\rightarrow\xi
	\end{equation}
	where $\tau\in\mathbb{R}$ and $\xi\in\mathbb{R}^3$ are fixed number and vector respectively. Rewrite
	\begin{equation}\label{eq:12191}
		\vec{\gamma}^{(k+1)}_n(-t^{(k+1)}_n,\cdot-x^{(k+1)}_n)=\vec{S}(t^{(j)}_n-t^{(k+1)}_n)\vec{\gamma}^{(k+1)}_n(-t^{(j)}_n,\cdot-x^{(k+1)}_n).
	\end{equation}
	By the finiteness of the first limitation in \eqref{eq:12190} and the strong continuity of $\vec{S}(t)$ in the time variable, it follows from this identity that in order to obtain the contraction, we shall show that $\vec{\gamma}^{(k+1)}_n(-t^{(j)}_n,\cdot-x^{(k+1)}_n)$ converges weakly to zero. For this, we use \eqref{eq:12192} to rewrite
	\begin{align}\label{eq:12193}
		\vec{\gamma}^{(k+1)}_n(-t^{(j)}_n,\cdot-x^{(k+1)}_n) =\vec{\gamma}^{(k)}_n(-t^{(j)}_n,\cdot-x^{(j)}_n+x^{(j)}_n-x^{(k+1)}_n) - \vec{V}^{(k)}_n(-t^{(j)}_n+t_n^{(k)},\cdot-x^{(k+1)}_n).
	\end{align}
	Combining the finiteness of the second limitation in \eqref{eq:12190} with the induction assumption \eqref{eq:03154} in the $k$-step, we see that the first term on the right hand side of \eqref{eq:12193} converges weakly to zero. Thanks to the induction assumption \eqref{eq:03155} in the $k$-th step and the uniform boundedness of $\{\vec{V}^{(k)}_n\}_n$, we can use dispersive estimate to see that the second term converges weakly to zero as well. Putting these two points together finishes the proof of \eqref{eq:03154} and \eqref{eq:03155} up to the $(k+1)$-th step.
	
	By the induction principle, we finish the proof of \eqref{eq:03154} and \eqref{eq:03155}.
	
	The assertion \eqref{eq:03155} in turn {implies the asymptotic orthogonal decomposition of energy as in \eqref{eq:03156}}.
	Substituting \eqref{eq:03157} for each $j$ into \eqref{eq:03156}, we obtain
	\begin{equation}
		\sum_j\left(\nu^{(j)}\right)^2\lesssim \sum_{j}\left\|\vec{V}^{(j)}\right\|^2_{(L^\infty\mathcal{H})^2}\lesssim \sup_{n}\left\|\vec{U}^{(j)}\right\|_{(L^\infty\mathcal{H})^2}^2<\infty,
	\end{equation}
	where the last inequality follows from the assumption.
	This in turn implies
	\begin{equation}
		\nu^{(j)}\rightarrow_{j\rightarrow\infty}0
	\end{equation}
	which proves \eqref{eq:lin:220} by definition of $\nu^{(j)}$. This completes the proof of Proposition \ref{prop:lin:prof}.
	\end{proof}

\section{Perturbation lemma}\label{sec:pert}
	In this section, we will adapt the perturbation lemma for the Klein-Gordon equation to our system \eqref{eq:skg:int}. As seen in next section, this result is very effective in analyzing the interaction of nonlinear profiles (the modifications of its linear one via local Cauchy theory).
\begin{lemma}\label{lemma:1}
	Let $I\subset\mathbb{R}$ be an interval, containing $t_0$. Let
	\begin{equation*}
		U,V\in\left(C(I,H^1)\cap C^1(I,L^2)\right)^2
	\end{equation*}
	satisfy for some $B>0$
	\begin{align}\label{eq:03158}
		\left\|V\right\|_{\left(L^3_t(I,L^6_x)\right)^2}\leq B
	\end{align}
	and for some positive number $\epsilon$ to be specified
	\begin{align}\label{eq:03159}
		\sum_{\#\in\{U,V\}}\left\|\mathrm{eqn}(\#)\right\|_{\left(L^1_t(I,L^2_x)\right)^2} +\left\|S(\cdot-t_0)\left(\vec{U}-\vec{V}\right)(t_0)\right\|_{\left(L^3_t(I,L^6_x)\right)^2}<\epsilon.
	\end{align}
	Then there exists a small positive real number $\epsilon_0=\epsilon_0(B)$, so that if $\epsilon<\epsilon_0$, one has
	\begin{equation}
		\left\|\vec{U}-\vec{V}-\vec{S}(\cdot-t_0)\left(\vec{U}-\vec{V}\right)(t_0)\right\|_{L^\infty_t(I,\mathcal{H}\times\mathcal{H})}+\left\|U-V\right\|_{\left(L^3_t(I,L^6_x(\mathbb{R}^3))\right)}\leq C_0(B)\epsilon.
	\end{equation}
	In particular, there holds
	\begin{equation}
		\left\|U\right\|_{\left(L^3_t(I,L^6_x(\mathbb{R}^3))\right)^2}<\infty.
	\end{equation}
\end{lemma}
\begin{proof} The proof here precedes in the same lines as \cite{nakanishischlag2011}.
	For simplicity, we set $Z(I):=\left(L^3_t(I,L^3_x)\right)^2$ for any time interval $I$. Denote also
	\begin{equation}
		W:=U-V,\ \ e:=\left(\square+1\right)\left(U-V\right)-\mathcal{N}(U)+\mathcal{N}(V).
	\end{equation}
	
	Let $\delta_0>0$ be specified later on. Thanks to the assumption \eqref{eq:03158}, we partition the interval $I_+:=I\cap [0,\infty)$ by setting
	\begin{equation*}
		t_0<t_1<t_2<\cdots<t_n\leq\infty;\ \ I_j:=(t_j,t_{j+1}),\ \ \forall j=0,1,\dots,n-1; I\cap[t_0,\infty)=(t_0,t_n)
	\end{equation*}
	such that
	\begin{equation}\label{eq:031510}
		\left\|V\right\|_{Z(I_j)}\leq \delta_0,\ \ \forall j=0,1,\dots,n-1.
	\end{equation}	
	It then follows that $n\leq C(B,\delta_0)$. For each $j=0,1,2,\dots,n-1$, we denote the linear evolutions
	\begin{equation}
		\vec{W}_j(t):=\vec{S}(t-t_j)\vec{W}(t_j).
	\end{equation}
	By the uniqueness of solutions, we have the following versions of Duhamel equality
	\begin{align*}
		W(t)&= W_j(t)+\int_{t_0}^t\frac{\sin\left((t-s)\langle\nabla\rangle\right)}{\langle\nabla\rangle}\left(\square+1\right)\left({U}-{V}\right)(s)ds\\
		&=W_j(t)+\int_{t_0}^t\frac{\sin\left((t-s)\langle\nabla\rangle\right)}{\langle\nabla\rangle}\left(e+\mathcal{N}(V+W)-\mathcal{N}(V)\right)(s)ds
	\end{align*}
	for each such $j$.
	We next use Strichartz estimate, together with the assumption \eqref{eq:03159} and the partitioning condition \eqref{eq:031510}, to obtain
	\begin{align}\label{eq:031511}
		\left\|W-W_0\right\|_{Z(I_0)}&\lesssim \left\|e+\mathcal{N}(V+W)-\mathcal{N}(V)\right\|_{\left(L^1_t(I_0,L^2_x)\right)^2}\nonumber\\
		&\lesssim \left\|\mathcal{N}(V+W)-\mathcal{N}(V)\right\|_{\left(L^1_t(I_0,L^2_x)\right)^2}+\left\|e\right\|_{\left(L^1_t(I_0,L^2_x)\right)^2}\\
		&\leq C_1\left(\delta_0^2+\left\|W\right\|^2_{Z(I_0)}\right)\left\|W\right\|_{Z(I_0)}+C_1\epsilon\nonumber
	\end{align}
	for some constant $C_1\geq 1$.
	
	From the assumption \eqref{eq:03159}, we see
	\begin{equation}\label{eq:031512}
		\left\|W_0\right\|_{Z(I_0)}\leq \epsilon.
	\end{equation}
	Substituting this into \eqref{eq:031511} yields
	\begin{equation}
			\left\|W\right\|_{Z(I_0)}\leq C_1\left(\delta_0^2+\left\|W\right\|^2_{Z(I_0)}\right)\left\|W\right\|_{Z(I_0)}+2C_1\epsilon.
	\end{equation}
	{We take $\epsilon_0$ sufficiently small so that $C_1(8C_1\epsilon)^2\leq \frac{1}{4}$ for all $0<\epsilon<\epsilon_0$.}
	Using the continuity argument on the right end point of the interval $I_0$ (by putting $\delta_0$ small), we get
	\begin{equation}\label{eq:031513}
		\left\|W\right\|_{Z(I_0)}\leq 4C_1\epsilon.
	\end{equation}
	{In the infinite endpoint case (that is, $t_n=\infty$), we shall use a truncation argument.}
	
	Using Duhamel formula issued from the initial data, we subtract $W_0(t)$ from $W_1(t)$ to obtain
	\begin{align*}
			W_1(t)-W_0(t)&= \int_{t_0}^{t_1}\frac{\sin\left((t-s)\langle\nabla\rangle\right)}{\langle\nabla\rangle}\left(\square+1\right)\left({U}-{V}\right)(s)ds\\
		&=\int_{t_0}^{t_1}\frac{\sin\left((t-s)\langle\nabla\rangle\right)}{\langle\nabla\rangle}\left(e+\mathcal{N}(V+W)-\mathcal{N}(V)\right)ds.
	\end{align*}
	By the Strichartz estimate on $I_0$, we obtain
	\begin{align}\label{eq:031514}
		\left\|W_1-W_0\right\|_{Z(I)}&\lesssim \int_{t_0}^{t_1}\left\|e+\mathcal{N}(V+W)-\mathcal{N}(V)\right\|_{L^2_x\times L^2_x}(s)ds\nonumber\\
		&\leq C_1\left(\delta_0^2+\left\|W\right\|^2_{Z(I_0)}\right)\left\|W\right\|_{Z(I_0)}+C_1\epsilon
	\end{align}
	where $C_1$ is the same constant as in \eqref{eq:031511}. Together with \eqref{eq:031512} and \eqref{eq:031513}, this last estimate implies
	\begin{equation}
		\left\|W_1\right\|_{Z(I)}\leq 4C_1\epsilon.
	\end{equation}

	We can repeat the above argument, to obtain for each $0\leq j<n$ that
	\begin{equation}
		\left\|W-W_j\right\|_{Z(I_j)}+\left\|W_{j+1}-W_j\right\|_{Z(I)}\leq 4C_1\epsilon,
	\end{equation}
	which in turn implies
	\begin{equation}
			\left\|W\right\|_{Z(I_j)}+\left\|W_{j+1}\right\|_{Z(I)}\leq C(j)\epsilon. \ \ \forall j=1,2,\dots,n-1.
	\end{equation}
	Finally we use triangle inequality to obtain
	\begin{equation}
		\left\|U-V\right\|_{\left(L^3_t(I,L^6_x(\mathbb{R}^3))\right)^2}\leq \sum_{1\leq j<n}\left\|W\right\|_{Z(I_j)}\leq C(B)\epsilon
	\end{equation}
	for some $C(B)>0$. To show the remaining part of the asserted estimate, {we shall repeat the argument in \eqref{eq:031511} and \eqref{eq:031514} with $L^\infty_t(I_0,\mathcal{H}\times \mathcal{H})$ in place of $Z(I_0)$.}
\end{proof}

\section{Proof of main result}\label{sec:proof}
	We argue by contradiction. Assume there existed a positive number $E_\ast<J[Q_1,Q_2]$, for which one could find a sequence $\left\{\vec{U}_n\right\}\subset \mathcal{PS}^+$, satisfying
	\begin{equation}
		E[\vec{U}_n]\nearrow E_\ast
	\end{equation}
	and
	\begin{equation}
		\left\|U_n\right\|_{\left(L^3_t(\mathbb{R},L^6_x(\mathbb{R}^3))^2\right)^2}\rightarrow_{n\rightarrow\infty}\infty.
	\end{equation}

\subsection{Extraction of a critical element}
	In this subsection,
	{we aim at extracting a critical element $\vec{U}_\ast$ with the following property: there exists $(t_n,x_n)\in\mathbb{R}\times\mathbb{R}^3$ satisfying for each $n$
			\begin{equation}
				\vec{U}_n=\vec{U}_\ast(\cdot+t_n,\cdot+x_n)+\vec{r}_n
			\end{equation}
			with
			\begin{enumerate}[(a)]
				\item $U_\ast$ is a global strong solution to \eqref{eq:skg:int};
				\item $\|\vec{r}_n\|_{L^\infty_t(\mathcal{H}\times\mathcal{H})}\rightarrow_{n\rightarrow\infty}0$;
				\item $\left\|U_\ast\right\|_{\left(L^3_tL^6_x\right)^2}=\infty$;
				\item $K_0[U_\ast]\geq 0$.
			\end{enumerate}
	}
	
	Since $\left\{\vec{U}_n\right\}\subset\mathcal{PS}^+$ is a sequence of solutions to \eqref{eq:skg:int} with energy bounded uniformly {by the positive number $E_\ast$ that is strictly smaller than the ground state energy}, we use the variational characterization of the ground state solution to conclude that the free energy $E_0[\vec{U}_n]$ is uniformly bounded. This in turn allows for {applying linear profile decomposition (Proposition \ref{prop:lin:prof})} to the sequence $\left\{\vec{S}(t)\vec{U}_n(0)\right\}$, withdrawing profiles $\vec{V}^{(j)}$ and $(t^{(j)}_n,x_n^{(j)})\in\mathbb{R}\times\mathbb{R}^3$ that satisfy the following properties:
	\begin{enumerate}[(i)]
		\item we have the asymptotic decomposition for each integer $k\geq 1$
		\begin{equation}
			\vec{U}_n=\sum_{0\leq j<k} \vec{V}^{(j)}(t+t^{(j)}_n,x+x^{(j)}_n)+\vec{\gamma}_n^{(k)}
		\end{equation}
		\item for $j<k$, we have the space-time separation
		\begin{equation}\label{eq:12210}
			\left|t^{(j)}_n-t^{(k)}_n\right|+\left|x^{(j)}_n-x^{(k)}_n\right|\rightarrow_{n\rightarrow\infty}\infty
		\end{equation}
		\item the asymptotic vanishing of Strichartz norms of $\gamma_n^{(k)}$
		\begin{equation}\label{eq:04240}
			\limsup_{n\rightarrow\infty}\left\|\gamma_n^{(k)}\right\|_{\left(L^\infty_tL^p_x\cap L^3_tL^6_x\right)^2}\rightarrow_{k\rightarrow\infty}0
		\end{equation}
		for each $p\in(2,6)$.
		\item the asymptotic orthogonality of the energy
		\begin{equation}\label{eq:03170}
			\left\|\vec{U}_n\right\|^2_{\mathcal{H}\times\mathcal{H}}=\sum_{0\leq j<k}\left\|\vec{V}^{(j)}\right\|^2_{\mathcal{H}\times\mathcal{H}} +\left\|\vec{\gamma}^{(k)}_n\right\|^2_{\mathcal{H}\times\mathcal{H}} +\mathbf{o}(1)
		\end{equation}
		as $n$ tends to infinity.
	\end{enumerate}

	Since our problem is a semilinear one, we can not use the above (linear) decomposition directly. In order to resolve our present problem, we shall replace all the linear waves by the corresponding nonlinear evolutions of \eqref{eq:skg:int}, as is done as follows.
	
	Set $t^{(j)}_\infty:=\lim_{n\rightarrow_{n\rightarrow\infty}}t^{(j)}_n\in\bar{\mathbb{R}}$. Since energies of $\{\vec{V}^{(j)}\}$ are uniformly bounded away from $J[Q_1,Q_2]$, we can use the local well-posedness theory (Proposition \ref{prop:cauchy}) to find a family $\{\vec{U}^{(j)}\}$ and some positive number $\tau_0>0$ such that
	\begin{enumerate}[(a)]
		\item for each $j$, $U^{(j)}$ solves \eqref{eq:skg:int} on the interval $[t^{(j)}_n-\tau_0,t^{(j)}_n+\tau_0]$, and
		\item there holds the limitation
		\begin{equation}\label{eq:03171}
			\left\|\vec{V}^{(j)}(t,\cdot)-\vec{U}^{(j)}(t,\cdot)\right\|_{\mathcal{H}\times\mathcal{H}}\rightarrow_{t\rightarrow t^{(j)}_\infty}0.
		\end{equation}
	\end{enumerate}
	What's more, for small time $t$, we also have
	\begin{equation}\label{eq:0316}
		\vec{U}_n(t)=\sum_{0\leq j<k}\vec{U}^{(j)}(t+t^{(j)}_n,\cdot + x^{(j)}_n) +\vec{\gamma}^{(k)}_n(t)+\vec{\eta}^{(k)}_n(t)
	\end{equation}
	where
	\begin{equation}
		\left\|\vec{\eta}^{(k)}_n(0)\right\|_{\mathcal{H}\times\mathcal{H}}\rightarrow_{n\rightarrow\infty}0.
	\end{equation}
		In particular, there holds
	\begin{equation}
		\left\|\vec{U}^{(j)}(t^{(j)}_n)-\vec{V}^{(j)}(t_n^{(j)})\right\|_{\mathcal{H}\times\mathcal{H}}\rightarrow_{n\rightarrow\infty}0.
	\end{equation}
	
	\begin{remark}
		Assume for each $j$, $\vec{U}^{(j)}$ is a global solution. Then one recovers the asymptotic superposition principle for our nonlinear equation \eqref{eq:skg:int} thanks to the space-time separation \eqref{eq:12210}.	
	\end{remark}

	We first show the asymptotic decomposition of $(L_x^4\times L_x^4)$-norm of the sequence $\{U_n(0)\}$.
	\begin{proposition}\label{claim:3} With the notations as above, there holds
		\begin{equation}
			\left\|U_n(0)\right\|^4_{L^4_x\times L^4_x}=\sum_{0\leq j<k} \left\|V^{(j)}(t^{(j)}_n)\right\|^4_{L^4_x\times L^4_x} +\mathbf{o}(1)
		\end{equation}
		as $k,n$ tend to infinity.
	\end{proposition}
	\begin{proof}
		We prove the result by integrating the fourth order power of both sides of \eqref{eq:0316} and consider all the possible interactions between terms from the right hand side of \eqref{eq:0316}.
		\begin{enumerate}
			\item It follows directly from \eqref{eq:04240} with $p=4$ that
			\begin{align}
				\left\|\gamma^{(k)}_n\right\|_{L^4\times L^4}^4\rightarrow_{k,n\rightarrow\infty} 0.
			\end{align}
			\item For the crossing terms between $V^{(j)}$ and $V^{(j')}$ for different $j$ and $j'$, we only treat the first component as follows. By symmetry, the terms we need to consider are
			\begin{equation}\label{eq:03161}
				\left\|\left[V^{(j)}_n(t^{(j)}_n,\cdot+x^{(j)}_n)\right]^2\left[V^{(j')}_n(t^{(j')}_n,\cdot+x^{(j')}_n)\right]^2\right\|_{L^1_x}
			\end{equation}
			and
			\begin{equation}\label{eq:03162}
				\left\|V^j_n(t^{(j)}_n,\cdot+x^{(j)}_n)\left[V^{(j')}_n(t^{(j')}_n,\cdot+x^{(j')}_n)\right]^3\right\|_{L^1_x}.
			\end{equation}
			If $\left|t^{(j)}_n-t^{(j')}_n\right|$ tends to infinity as $n$ tends to infinity, then both \eqref{eq:03161} and \eqref{eq:03162} tend to zero and $n$ tends to infinity, thanks to the dispersive estimate. While if  $\left|t^{(j)}_n-t^{(j')}_n\right|$ remains bounded as $n$ tends to infinity, we shall have $\left|x^{(j)}_n-x^{(j')}_n\right|$ tends to infinity.  {In this case, the assertion follows from the basic property of integration theory (the simple case is: the interacting two terms are of compact support, then as $n$ tends to infinity, the distance of these two supporting set is pushed away from each other and hence they cannot see each other, resulting the integration to be zero).}
			
			\item For the crossing terms between $\vec{V}_n^{(j)}$ and $\vec{\gamma}^{(k)}_n$, we shall exploit the fact that energies of $\vec{V}^{(j)}$ remains bounded and \eqref{eq:04240}.
		\end{enumerate}
		This completes the proof of Proposition \ref{claim:3}.
	\end{proof}

	{We next derive the positivity of energies of each bubble $\vec{U}^{(j)}$ and $\vec{\gamma}^{(k)}_n$}.
		Using the definition of the energy functional $E$ together with Proposition \ref{claim:3}, \eqref{eq:03171} and \eqref{eq:03170}, we obtain
		\begin{equation}\label{eq:03176}
			E_\ast+\mathbf{o}(1)>E\left[\vec{U}_n\right]=\sum_{j<k}E\left[\vec{U}^{(j)}\right] +E[\vec{\gamma}^{(k)}_n]+\mathbf{o}(1)
		\end{equation}
		and
		\begin{equation}
			K_0[U_n(0)]=\sum_{j<k}K_0[{U}^{(j)}(t^{(j)}_n)] +K_0[\gamma^{(k)}_n(0)] +\mathbf{o}(1)
		\end{equation}
		as $n$ tends to infinity.
		
		Thanks to $\vec{U}_n(0)\in\mathcal{PS}^+$, we have $K_0[U_n(0)]\geq 0$. We now do a series of estimates
		\begin{align}
			J[Q_1,Q_2]&>E_\ast+\mathbf{o}(1)\nonumber\\
			&\geq E[\vec{U}_n(0)]-\frac{K_0[U_n(0)]}{4}\nonumber\\
			&=G_0[U_0(0)]+\frac{1}{4}\left\|\partial_tU_n(0)\right\|^2_{L^2_x\times L^2_x}\label{eq:03172}\\
			&=\frac{1}{2}E_0[\vec{U}_n(0)]\label{eq:03174}\\
			&=\sum_{j<k}\frac{1}{2}E_0[\vec{U}^{(j)}(t^{(j)}_n)]+\frac{1}{2}E_0[\vec{\gamma}^{(k)}_n(0)]+\mathbf{o}(1)\label{eq:03175}\\
			&\geq \sum_{j<k}G_0[U^{(j)}(t^{(j)}_n)]+G_0[\gamma^{(k)}_n(0)]+\mathbf{o}(1)\label{eq:03173}
		\end{align}
		asymptotically as $n$ tends to infinity. Here in both \eqref{eq:03172} and \eqref{eq:03173}, we used the definition of $G_0$, in \eqref{eq:03174} the definition of $E_0$ and in \eqref{eq:03175} the orthogonal relation \eqref{eq:0316}. Thanks to the non-negativeness of the functional $G_0$, one has that for each $j$ and $k$
		\begin{equation}
			G_0[U^{(j)}(t^{(j)}_n)]\leq E_\ast < J[Q_1,Q_2]\ \ \mathrm{and}\ \ G_0[\gamma^{(k)}_n(0)]<J[Q_1,Q_2].
		\end{equation}
		It then follows from the {variational characterization of $(Q_1,Q_2)$} that
		\begin{equation}\label{eq:03177}
			K_0[U^{(j)}(t^{(j)}_n)]\geq 0\ \mathrm{and}\ K_0[\gamma_n^{(k)}(0)]\geq 0.
		\end{equation}
		From this, one concludes by using definitions of $K_0$ and $E$ that
		\begin{equation}
			E[\vec{U}^{(j)}]\geq 0\ \mathrm{and}\ E[\vec{\gamma}^{(k)}_n]\geq0.
		\end{equation}
		
		Substituting this into \eqref{eq:03176}, one gets
		\begin{equation}
			E[\vec{U}^{(j)}]\leq E_\ast,\ \forall j.
		\end{equation}
		{Note that this does not allow us to use the absurd assumption at beginning of our argument (the minimality of $E_\ast$) to conclude that each profile $U^{(j)}$ is a global solution to \eqref{eq:skg:int}. Nevertheless, there can be only one profile.}
		\begin{proposition}\label{claim:4}
			There can only be one nonvanishing profile, say $U^{(1)}$, and
			\begin{equation}
				\lim_{k\rightarrow\infty}\limsup_{n\rightarrow\infty} \left\|\vec{\gamma}^{(k)}_n\right\|_{\mathcal{H}\times\mathcal{H}}=0.
			\end{equation}
		\end{proposition}
		\begin{proof}
			We argue by contradiction. Assume that \begin{enumerate}[(i)]
				\item either there exist two non-vanishing profiles $\vec{U}^{(1)}$ and $\vec{U}^{(2)}$,
				\item or one has
				\begin{equation}\label{eq:03178}
					\lim_{k\rightarrow\infty}\limsup_{n\rightarrow\infty} 		\left\|\vec{\gamma}^{(k)}_n\right\|_{\mathcal{H}\times\mathcal{H}}>0.
				\end{equation}
			\end{enumerate}
			In the former case $(i)$, it follows from \eqref{eq:03177} that
			\begin{align}
				E[\vec{U}^{(j)}(t^{(j)}_n)]&\geq G_0[U^{(j)}(t^{(j)}_n)] +\frac{1}{4}\left\|\partial_tU^{(j)}(t^{(j)}_n)\right\|^2_{L^2\times L^2}\\
				&=\frac{1}{2}E_0[\vec{U}^{(j)}(t^{(j)}_n)]+\mathbf{o}(1)
			\end{align}
			as $n$ tends to infinity for $j=1$ and $2$. Since both linear energies are positive, this last inequality yields
			\begin{equation}
				E[\vec{U}^{(1)}]>0\ \mathrm{and}\ E[\vec{U}^{(2)}]>0.
			\end{equation}
			In the latter case $(ii)$, it follows from \eqref{eq:03178} that for all sufficiently large $n$ and $k$, one has
			\begin{equation}
				E[\vec{\gamma}^{(k)}_n]>\delta_0
			\end{equation}
			for some positive number $\delta_0$. Thus one infers from the asymptotic \eqref{eq:03170} that the energy of any profile $\vec{V}^{(j)}$ is strictly less than $E_\ast$.
			
			In summary, in each case, any nontrivial profile has its energy strictly less than $E_\ast$.
			It follows from the minimality of $E_\ast$ that each $U^{(j)}$ is indeed a global solution to \eqref{eq:skg:int} and
			\begin{equation}
				\left\|U^{(j)}\right\|_{L^3_tL^6_x\times L^3_tL^6_x}<\infty.
			\end{equation}
			
		To continue the proof, one will apply the perturbation lemma \ref{lemma:1} with
		\begin{enumerate}[(a)]
			\item $I=\mathbb{R}$,
			\item $\vec{U}(t)=\vec{U}_n(t)$ and
			\item $\vec{V}(t)=\sum_{0\leq j<k}\vec{U}^{(j)}(t+t^{(j)}_n,\cdot+x^{(j)}_n)$.
		\end{enumerate}
		It is obvious that this choice of $U$ and $V$ satisfies the smoothness assumption in Lemma \ref{lemma:1}. Using consecutively Minkowski's inequality and Strichartz estimates, we can find a positive constant $B$ (depending only on the initial energy) such that
		\begin{equation}\label{eq:031712}
			\limsup_{n\rightarrow\infty}\left\|\sum_{j<k}V^{(j)}(\cdot+t^{(j)}_n,\cdot+x^{(j)}_n)\right\|_{\left(L^3_tL^6_x\right)^2}\lesssim \limsup_{n\rightarrow\infty}\sum_{j<k}\left\|\vec{V}^{(j)}\right\|_{\mathcal{H}\times \mathcal{H}}\leq B<\infty.
		\end{equation}
		It follows from the assumption that
		\begin{equation}\label{eq:031711}
			\left\|\square U+U-\mathcal{N}(U)\right\|_{L^1_t(\mathbb{R},L^2\times L^2)}=0.
		\end{equation}
		Doing the same for $V$, and using ideas in the proof of Proposition \ref{claim:3}, we obtain
		\begin{equation}\label{eq:031710}
			\left\|\square V+V-\mathcal{N}(V)\right\|_{L^1_t(\mathbb{R},L^2\times L^2)}=\mathrm{off-diagonal}\rightarrow_{n\rightarrow\infty}0.
		\end{equation}
		For the linear evolution, we use Strichartz estimates once again to get
		\begin{equation}\label{eq:03179}
			\left\|S(t)\left(\vec{U}(0)-\vec{V}(0)\right)\right\|_{\left(L^3_tL^6_x\right)^2}=\left\|S(t)\vec{\gamma}^{(k)}_n(0)\right\|_{\left(L^3_tL^6_x\right)^2}\rightarrow_{k,n\rightarrow\infty}0
		\end{equation}
		where the last limit follows from the construction of $\vec{\gamma}^{(k)}_n$ (see for instance \eqref{eq:lin:220}).
		Theses estimate \eqref{eq:031712}-\eqref{eq:03179} show that the choice $I=\mathbb{R}$, $U$ and $V$ verifies the conditions in Lemma \ref{lemma:1}, at least for sufficiently large $k$ and $n$. Then one applies this lemma to conclude
		\begin{equation}
			\limsup_{n\rightarrow\infty}\left\|U_n\right\|_{L^3_tL^6_x\times L^3_tL^6_x}<+\infty.
		\end{equation}
		But this contradicts to the assumption on $\vec{U}_n$. Thus there can only one profile, completing the proof of Proposition \ref{claim:4}.
	\end{proof}

\subsection{Compactness about the critical element}
	By Proposition \ref{claim:4}, we have
	\begin{equation}
		E[\vec{U}^{(1)}] = E_\ast\ \ \ \ K_0[U^{(1)}]\geq 0\  \mathrm{and}\ \left\|U^{(1)}\right\|_{\left(L^3_tL_x^6\right)^2}=\infty.
	\end{equation}
	In the following, we set
	\begin{equation}
		\vec{U}_\ast=\left((U^1_{\ast},\partial_tU^1_{\ast}),(U^2_{\ast},\partial_tU^2_{\ast})\right)^{\mathsf{T}}:=\vec{U}^{(1)}.
	\end{equation}
	For a vector-valued function $x_0:\mathbb{R}\ni t\mapsto x_0(t)\in\mathbb{R}^3$ to be specified later on, we define
	\begin{equation*}
		\mathcal{K}_{\pm}:=\left\{ \vec{U}_\ast(x+x_0(t),t):0\leq \pm t<\infty\right\}.
	\end{equation*}
	
	{Our goal in this subsection is to show that both sets $\mathcal{K}_{\pm}$ are precompact for suitable choices of $x_0(t)$.}
	
	\begin{proposition}\label{lemma:2}
		There exists a vector-valued function $x_0:[0,\infty) \mapsto \mathbb{R}^3$ such that, for any $\epsilon>0$, one can find a number $R(\epsilon)\in(0,\infty)$ with the property that
		\begin{equation}\label{eq:1026}
			\int_{[|x-x_0(t)|>R(\epsilon)]}\left[\left|\nabla U_\ast\right|^2+\left|U_\ast\right|^2+\left|\partial_tU_\ast\right|^2\right]dx<{\color{red}\epsilon}
		\end{equation}
		holds for all $t\geq 0$.
	\end{proposition}
	\begin{proof} We first show a weaker statement:
			for any $\epsilon>0$, there exist $x_{0,\epsilon}(t)$ and $R(\epsilon)$ such that \eqref{eq:1026} holds with $x_{0,\epsilon}(t)$ in place of $x_0(t)$.

		We argue by contradiction, assuming this is false. Then for some $\epsilon>0$, we take a sequence $\{t_n\}_n\subset(0,\infty)$ with
		\begin{equation}\label{eq:03280}
			\inf_{y\in\mathbb{R}^3}	\int_{[|x-y|>n]}\left[\left|\nabla U_\ast\right|^2+\left|U_\ast\right|^2+\left|\partial_tU_\ast\right|^2\right](t_n,x)dx>\epsilon.
		\end{equation}
	
		Thanks to the fact that $U_\ast$ is the critical element, we use Proposition \ref{prop:lin:prof} to decompose $$\vec{U}_\ast(t_n)=\vec{V}(\tau_n,\cdot+\xi_n)+\vec{r}(0)$$
		where both $V\neq 0$ and $r_n$ are free waves, satisfying $$\left\|\vec{r}_n(0)\right\|_{\mathcal{H}\times\mathcal{H}}\rightarrow_{n\rightarrow\infty}0.$$
		
		\begin{claim}\label{claim:03280}
			$\{\tau_n\}$ is a bounded sequence.
		\end{claim}
		\begin{proof}[Proof of Claim \ref{claim:03280}]
			Since $U_\ast$ is the critical element, we may first assume
			\begin{equation}\label{eq:10260}
				\left\|U_\ast\right\|_{\left(L^3_tL_x^6([0,\infty),\mathbb{R}^3)\right)^2}=\infty
			\end{equation}
		
			If $\tau_n\rightarrow\infty$, then $\left\|V(\cdot+\tau_n,\cdot+\xi_n)\right\|_{\left(L^3_tL_x^6([0,\infty),\mathbb{R}^3)\right)^2}=\left\|V(\cdot+\tau_n,\cdot)\right\|_{\left(L^3_tL_x^6([0,\infty),\mathbb{R}^3)\right)^2}$ tends to $0$ as $n$ approaches infinity. This in turn allows us to use local well-posedness theory to infer that $\left\|U_\ast(\cdot+t_n)\right\|_{\left(L^3_tL_x^6([0,\infty),\mathbb{R}^3)\right)^2}<\infty$ for all $n$ that is sufficiently large, which contradicts \eqref{eq:10260}.
			
			If $\tau_n\rightarrow-\infty$, then $\left\|V(\cdot+\tau_n,\cdot+\xi_n)\right\|_{\left(L^3_tL_x^6((-\infty,0],\mathbb{R}^3)\right)^2}=\left\|V(\cdot+\tau_n,\cdot)\right\|_{\left(L^3_tL_x^6((-\infty,0],\mathbb{R}^3)\right)^2}$ tends to $0$ as $n$ approaches infinity. This in turn implies that for some constant $B>0$, there holds
			\begin{equation*}
				\left\|U_\ast(\cdot+t_n)\right\|_{\left(L^3_tL_x^6((-\infty,0],\mathbb{R}^3)\right)^2}<B
			\end{equation*}
			for all sufficiently large $n$. Now sending $n$ to infinity, one sees the contradiction to \eqref{eq:10260}
			
			Combining the above two sides, one finishes the proof of Claim \ref{claim:03280}.
		\end{proof}
		
		From this claim, we can assume $\tau_{n}\rightarrow_{n\rightarrow\infty}\tau_{\infty}$ for some finite $\tau_\infty\in\mathbb{R}$. On the other hand, since $\vec{U}_\ast$ has finite energy, the linear energy of $\vec{U}_\ast(t_\infty)$ is small at spatial infinity. But this contradicts with \eqref{eq:03280}.
		
		Next we are going to remove the $\epsilon$-dependence of $x_{0,\epsilon}$. Since $\vec{U}_\ast\in\mathcal{PS}^+$, we have
		\begin{equation}
			\forall t\in\mathbb{R},\ \int_{\mathbb{R}^3}\left[\left|\nabla U_\ast\right|^2+\left|U_\ast\right|^2+\left|\partial_tU_\ast\right|^2\right]dx \sim E[\vec{U}_\ast].
		\end{equation}
		Fix some $\epsilon_0>0$, so that for $R_0:=R(\epsilon_0)$, we have that for any $t\in\mathbb{R}$, the LHS of \eqref{eq:1026} has a small portion of the total integral with $x_0(t):=x_{0,\epsilon_0}(t)$. For any $\epsilon\in(0,\epsilon_0)$, we take a number $R(\epsilon)$ and a vector-valued function $x_{0,\epsilon}$ so that two balls $B\left(x_{0,\epsilon}(t),R(\epsilon)\right)$ and $B(x_0(t),R_0)$ intersects. Thus if we replace $R(\epsilon)$ by $3R(\epsilon)$, we can replace $x_{0,\epsilon}(t)$ with $x_0(t)$. This completes the proof of Proposition \ref{lemma:2}.
	\end{proof}
	
	\begin{remark}\label{rem:5}
		It follows from the proof of Proposition \ref{lemma:2} that upon fixing $\epsilon_0$ it is that the relative size to $x_0(t)$ that affects the choice of $x_{0,\epsilon}$. Thus we can locally (in time) mollify it so that the two balls defined as in the proof still intersects.
	\end{remark}
	
	\begin{corollary}\label{cor:1}
		With $x_0(t)$ as in Proposition \ref{lemma:2}, $\mathcal{K}_+$ is precompact in $\mathcal{H}\times\mathcal{H}$.
	\end{corollary}
	\begin{proof}
		Suppose this fails, then there exist $\delta>0$ and a sequence $\left\{t_n\right\}_{n\geq 1}$ such that
		\begin{equation}\label{eq:03281}
			\forall n\neq m,\ \ \left\|\vec{U}_\ast(t_n,\cdot+x_0(t_n))-\vec{U}_\ast(t_m,\cdot+x_0(t_m))\right\|_{\mathcal{H}\times\mathcal{H}}>\delta.
		\end{equation}
	
		Recalling that
		\begin{equation}\label{eq:03282}
			\vec{U}_\ast(t_n,\cdot+x_0(t_n))=\vec{V}(\tau_n,\cdot+\xi_n)+\vec{r}_n(0),
		\end{equation}
		we give
		\begin{claim}\label{claim:0328}
			$\{\xi_n\}\subset \mathbb{R}^3$ is a bounded sequence.
		\end{claim}
		\begin{proof}[Proof of Claim \ref{claim:0328}]
			It follows from Lemma \ref{lemma:2} that for $R(\epsilon)\gg1$, we have
			\begin{equation}
				\int_{[|x-x_0(t)|>R(\epsilon)]}\left[\left|\nabla U_\ast\right|^2+\left|U_\ast\right|^2+\left|\partial_tU_\ast\right|^2\right](t_n,x)dx<\epsilon.
			\end{equation}
			Substituting \eqref{eq:03282} into this expression, we get for $n\gg 1$
			\begin{equation}
				\int_{[|x-x_0(t)|>R(\epsilon)]}\left[\left|\nabla V\right|^2+\left|V\right|^2+\left|\partial_tV\right|^2\right](\tau_n,x+\xi_n)dx\leq 2\epsilon
			\end{equation}
			which is equivalent to
			\begin{equation}\label{eq:03283}
				\int_{[|x-\xi_n-x_0(t)|>R(\epsilon)]}\left[\left|\nabla V\right|^2+\left|V\right|^2+\left|\partial_tV\right|^2\right](\tau_n,x)dx\leq 2\epsilon.
			\end{equation}
			Now if $|\xi_n|$ tends to infinity, it follows from the finiteness of linear energy of $V$ that
			\begin{equation}
				\int_{[|x-\xi_n-x_0(t)|\leq R(\epsilon)]}\left[\left|\nabla V\right|^2+\left|V\right|^2+\left|\partial_tV\right|^2\right](\tau_n,x)dx
			\end{equation}
			is going to $0$ as $n$ tends to infinity. This forces the left hand side of \eqref{eq:03283} to be equaling the linear energy of $V$, which is nonzero and makes this equality impossible if $\epsilon$ is taken to be sufficiently small.
		\end{proof}
		By Claims \ref{claim:03280} and \ref{claim:0328}, we can assume
		\begin{equation}
			\xi_n\rightarrow_{n\rightarrow\infty}\xi_\infty,\ \mathrm{and}\ \tau_n\rightarrow_{n\rightarrow\infty}\tau_\infty
		\end{equation}
		for some $\xi_\infty\in\mathbb{R}^3$ and $\tau_\infty\in\mathbb{R}$. Then by the absolute continuity of the integration, we see that the assumption \eqref{eq:03281} is impossible. This finishes the proof.
	\end{proof}

\subsection{The $0$-momentum property of the critical element} In this subsection, we are going to show
	\begin{lemma}\label{lemma:03281}
		The critical element $\vec{U}_\ast$ has zero momentum, that is
		\begin{equation}
			\mathcal{P}[\vec{U}_\ast]:=\int_{\mathbb{R}^3}\left[\sum_{j=1}^2\partial_tU^j_\ast\cdot \nabla U^1_\ast\right]dx=0,\ \ \forall t\geq 0.
		\end{equation}
	\end{lemma}

	\begin{proof}We argue by contraction, assuming
		\begin{equation}\label{eq:12215}
			\mathcal{P}[\vec{U}_\ast]\neq 0.
		\end{equation}
	We first give	
	\begin{claim}\label{claim:03281}
			Under the assumption \eqref{eq:12215}, there exists a Lorentz transform $L$ so that
			\begin{equation}
				E[\vec{U}_\ast\circ L] <E_\ast
			\end{equation}
			and
			\begin{equation}
				K_0[U_\ast\circ L]\geq 0.
			\end{equation}
	\end{claim}
	\begin{proof}[Proof of Claim \ref{claim:03281}]
		 Without loss of generality, we assume
		\begin{equation*}
			P_1[\vec{U}_\ast]\neq 0.
		\end{equation*}
		By the second item in Proposition \eqref{prop:lorentz}, we can find small $\lambda$ such that
		\begin{equation*}
			E\left[\overrightarrow{L^\lambda_1U}_\ast\right]=E\left[\overrightarrow{U}_\ast\right]+\lambda P_1[\vec{U}_\ast]+\mathcal{O}(\lambda^2)<E\left[\overrightarrow{U}_\ast\right]
		\end{equation*}
		If we had $K_0\left[\overrightarrow{L^\lambda_1U}_\ast\right]<0$, then ${L^\lambda_1U}_\ast$, and hence $U_\ast$would not be a global solution.
	\end{proof}
	
	By the minimality of $E_\ast$, $U_\ast\circ L$ exists globally and scatters. This implies that
	\begin{equation}\label{eq:03285}
		\left\|U_\ast\circ L\right\|_{L^3_tL^6_x\times L^3_tL^6_x}<\infty.
	\end{equation}
	
	It follows from Strichartz estimates that $V_\ast:=U_\ast\circ L$ satisfies
	\begin{equation}
		\left\|V_\ast\right\|_{L^{8/3}_tL^8_x\times L^{8/3}_tL^8_x} <\infty.
	\end{equation}
	Next we use H\"{o}lder's inequality to obtain
	\begin{equation}
		\left\|V_\ast\right\|_{L^4_{t,x}\times L^4_{t,x}}\leq \left\|V_\ast\right\|^{1/3}_{L^{\infty}_tL^2_x\times L^{\infty}_tL^2_x}\left\|V_\ast\right\|^{2/3}_{L^{8/3}_tL^8_x\times L^{8/3}_tL^8_x}<\infty,
	\end{equation}
	which implies
	\begin{equation}\label{eq:1111}
		\left\|U_\ast\right\|_{L^4_{t,x}\times L^4_{t,x}}<\infty
	\end{equation}

	Let $\vec{W}_T$ be the free skg wave with $\vec{W}_T(T)=\vec{U}_\ast(T)$. Then
	\begin{equation}\label{eq:03284}
		\sup_{T}\left\|W_T\right\|_{L^{8/3}_tL^8_x\times L^{8/3}_tL^8_x}\leq M<\infty
	\end{equation}
	for some $M>0$. Given $T<S$, it follows from Strichartz estimates that
	\begin{align}
		\left\|U_\ast-W_T\right\|_{\left(L^{8/3}_t((T,s),L^8_x)\right)^2}&\lesssim \left\|\mathcal{N}(U_\ast)\right\|_{\left(L^1_t((T,S),L^2_x)\right)^2}\\
		&\lesssim \left\|U_\ast\right\|_{\left(L^4_t((T,S),L^4_x)\right)^2}\left\|U_\ast\right\|^2_{\left(L^{8/3}_t((T,S),L^8_x)\right)^2}
	\end{align}
	Together with \eqref{eq:03284}, this implies
	\begin{equation}\label{eq:03286}
		\left\|U_\ast\right\|_{\left(L^{8/3}_t((T,S),L^8_x)\right)^2}\leq M+\left\|U_\ast\right\|_{\left(L^4_t((T,S),L^4_x)\right)^2}\left\|U_\ast\right\|^2_{\left(L^{8/3}_t((T,S),L^8_x)\right)^2}
	\end{equation}
	Using the continuity argument (on $S$, based on \eqref{eq:03285}), we can obtain
	\begin{equation}
		\sup_{S>T}\left\|U_\ast\right\|_{\left(L^{8/3}_t((T,S),L^8_x)\right)^2}\leq 2M
	\end{equation}
	for $T\gg 1$. Thanks to \eqref{eq:1111} and the finiteness of $M$, we can take $T$ even larger to obtain
	\begin{equation}
		M\cdot \left\|U_\ast\right\|_{\left(L^{4}_t((T,\infty),L^4_x)\right)^2}<\frac{1}{4}.
	\end{equation}
	Substituting these last two inequalities back into \eqref{eq:03286}, we see
	\begin{equation}\label{eq:03287}
		\left\|U_\ast\right\|_{\left(L^{8/3}_t((T,\infty),L^8_x)\right)^2}<\infty.
	\end{equation}
	On the other hand, since $U_\ast$ is a global solution with finite energy, we have
	\begin{equation}
		\left\|U_\ast\right\|_{\left(L^{8/3}_t((0,T),L^8_x)\right)^2}<\infty.
	\end{equation}
	We combine this inequality with \eqref{eq:03287} to get
	\begin{equation}
		\left\|U_\ast\right\|_{\left(L^{8/3}_t((0,\infty),L^8_x)\right)^2}<\infty.
	\end{equation}

	Using H\"{o}lder's inequality once again, we obtain
	\begin{equation}
		\left\|U_\ast\right\|_{L^3_{t}((0,\infty),L^6_x)\times L^3_{t}((0,\infty),L^6_x)}\leq \left\|U_\ast\right\|^{1/3}_{L^{4}_{t,x}\times L^{4}_{t,x}}\left\|U_\ast\right\|^{2/3}_{L^{8/3}_tL^8_x\times L^{8/3}_tL^8_x}<\infty.
	\end{equation}
	But this contradicts with the infinity of Strichartz norm of $\vec{U}_\ast$.
	\end{proof}

\subsection{Improvement of the growth of $x_0(t)$}
	By Corollary \ref{cor:1}, for any $\epsilon>0$, there exists $R_0(\epsilon)>0$ such that
	\begin{equation}\label{eq:12216}
		\int_{[|x-x_0(t)|>R_0(\epsilon)]}\left(\left|U_\ast\right|^2+\left|\partial_tU_\ast\right|^2+\left|U_\ast\right|^4\right)dx<\epsilon E[\vec{U}_\ast],\ \ \forall t\geq 0
	\end{equation}
	for our choice of the vector-valued map $x_0(t)$.
	{Our goal in this subsection is to improve the finite speed of propagation
	\begin{equation*}
		\limsup_{t\rightarrow\pm\infty}\frac{|x_0(t)|}{|t|}\leq 1.
	\end{equation*}
	}
	
	To achieve this goal, we will use a localized virial argument.
	Let $\chi\in C^\infty_{c}(B_2(0))$ that equals $1$ on the unit ball $B_1(0)$. Define $\chi_R(x):=\chi\left(\frac{x}{R}\right)$ for each $R>0$. We also set
	\begin{equation*}
		X_R(t):=\int_{\mathbb{R}^3}\chi_R(x)\cdot x\cdot e(t,x)dx
	\end{equation*}
	where
	\begin{equation*}
		e(t,x):=\frac{1}{2}\sum_{j=1}^2\left[\left|\nabla U^j_\ast\right|^2+\left|\partial_tU_\ast\right|^2\right]-\frac{1}{4}\left[\sum_{j=1}^2\left|U^j_\ast\right|^4+2\beta\left|U^1_\ast\right|^2\left|U^2_\ast\right|^2\right]
	\end{equation*}
	is the local energy density.
	
	Using the local conservation law
	\begin{equation*}
		\partial_te(t,x)=\mathrm{div}\left(\sum_{j=1}^{2}\partial_tU^j_\ast\cdot\nabla U^j_\ast\right)
	\end{equation*}
	together with Lemma \ref{lemma:03281}, we can compute
	\begin{equation*}
		\frac{d}{dt}X_R(t)=\int_{\mathbb{R}^3}\left(1-\chi_R(x)\right)\sum_{j=1}^2\partial_tU^j_\ast\cdot\nabla U^j_\ast dx -\int_{\mathbb{R}^3}x\cdot\nabla\chi_R \sum_{j=1}^2\partial_tU^j_\ast\nabla U^j_\ast dx.
	\end{equation*}
	Then by Cauchy-Schwarz inequality and the fact $\left|x\cdot \nabla\chi_R(x)\right|\lesssim 1$, we get
	\begin{equation}\label{eq:03290}
		\left|\frac{d}{dt}X_R(t)\right|\lesssim \int_{[|x|\geq R]}\left(\sum_{j=1}^2\left|\partial_tU^j_\ast\right|^2+\left|\nabla U^j_\ast\right|^2\right)dx.
	\end{equation}

	\begin{lemma}\label{lem:0517}
		If $0<\epsilon\ll 1$ and $R\gg R_0(\epsilon)$, then one has
		\begin{equation}
			\left|x_0(t)-x_0(0)\right|\leq R-R_0(\epsilon)
		\end{equation}
		for all $0<t<t_0$, where $t_0$ is of size $\sim \frac{R}{\epsilon}$.
	\end{lemma}

	\begin{proof}
		{By Remark \ref{rem:5}, we can assume $x_0$ is differentiable in time. For simplicity, we assume further that $x_0(0)=0$.}
		Let
		\begin{equation}
			t_0:=\inf\left\{t>0:\left|x_0(t)\right|\geq R-R_0(\epsilon)\right\}.
		\end{equation}
		Then for all $0<t<t_0$, we have
		\begin{equation}
			\left|x_0(t)\right|< R-R_0(\epsilon).
		\end{equation}
		For these $t$, the condition $|x|>R$ implies $|x-x_0(t)|\geq R_0(\epsilon)$. It then follows from \eqref{eq:03290} and \eqref{eq:12216} that
		\begin{equation}\label{eq:12217}
			\left|\frac{d}{dt}X_R(t)\right|\lesssim \epsilon E[\vec{U}_\ast],\ \ \forall t\in [0,t_0].
		\end{equation}
		On the other hand, for these $t$, we have
		\begin{align}\label{eq:03291}
			\left|X_R(t)\right|&\geq |x_0(t)|E[\vec{U}_\ast] -|x_0(t)|\left|\int_{\mathbb{R}^3}(1-\chi_R(x))e(t,x)dx\right|-\left|\int_{\mathbb{R}^3}(x-x_0(t))\chi_R(x)e(t,x)dx\right|\\
			&\geq |x_0(t)|E[\vec{U}_\ast](1-C\epsilon) -\int_{[|x-x_0(t)|>R_0(\epsilon)]}\left|x-x_0(t)\right|\chi_R(x)\left|e(t,x)\right|dx-R_0(\epsilon)\int\left|e(t,x)\right|dx\nonumber\\
			& \geq |x_0(t)|E[\vec{U}_\ast](1-C\epsilon) -CR_0(\epsilon)\epsilon E[\vec{U}_\ast]-R_0(\epsilon)\int_{\mathbb{R}^3}\left[e_0(t,x)+\mathcal{M}(U_\ast)\right]dx\nonumber
		\end{align}
		where $e_0$ is the free energy density and $\mathcal{M}(\vec{U}_\ast)$ is the corresponding potential energy density. Here in the second inequality, we used the fact that $|\int_{\mathbb{R}^3}(1-\chi_R(x))e(t,x)dx|\leq C\epsilon$ for large $R$, and in the third inequality the fact $|x-x_0(t)|\geq R_0(\epsilon)$ for $|x|>R$ together with \eqref{eq:12216}.  Thanks to $K_0[U_\ast]\geq 0$, we have
		\begin{equation*}
			\int_{\mathbb{R}^3}\left[e_0(t,x)+\mathcal{M}(U_\ast)\right]dx=3E[\vec{U}_\ast]-K_0[U_\ast]\leq 3E[\vec{U}_\ast].
		\end{equation*}
		Plugging this back to \eqref{eq:03291} gives
		\begin{equation*}
			\left|X_R(t)\right|\geq E[\vec{U}_\ast]\left[|x_0(t)|(1-C\epsilon)-CR\epsilon -3R_0(\epsilon)\right].
		\end{equation*}
		Letting $t$ tend to $t_0$, yields
		\begin{equation}\label{eq:03292}
			\left|X_R(t_0)\right|\geq \frac{1}{2}E[\vec{U}_\ast]\left(R-R_0(\epsilon)\right).
		\end{equation}
	
		Integrating \eqref{eq:12217} from $0$ to $t_0$ in the time variable, yields
		\begin{equation}\label{eq:12218}
			|X_R(t_0)|\leq |X_R(0)|+ t_0\epsilon E[\vec{U}_\ast].
		\end{equation} 	
		We split the integration domain in the definition of $X_R(0)$ into $[|x|\leq R_0(\epsilon)]$ and $[|x|>R_0(\epsilon)]$ to obtain
		\begin{equation}\label{eq:03293}
			\left|X_R(0)\right|\lesssim \left(R_0(\epsilon)+\epsilon R\right) E[\vec{U}_\ast].
		\end{equation}
		Inserting \eqref{eq:03292} and \eqref{eq:03293} into \eqref{eq:12218} gives
		\begin{equation}
			\frac{1}{2}\left(R-R_0(\epsilon)\right)\lesssim R_0(\epsilon) +\epsilon R+t_0\epsilon,
		\end{equation}
		which implies
		\begin{equation}
			t_0\sim \frac{R}{\epsilon}.
		\end{equation}
		This completes the proof.
	\end{proof}

\subsection{Arriving at the contradiction}
	In order to derive the final contradiction, we first give the following direct consequence of the precompact property of $\mathcal{K}_{\pm}$.
	\begin{claim}\label{claim:1}
		For any $\epsilon>0$, there exists a constant $C(\epsilon)$ such that
		\begin{equation}\label{eq:11170}
			\left\|U_\ast\right\|_{L^2\times L^2}\leq C(\epsilon)\left\|\nabla U_\ast(t)\right\|_{L^2\times L^2} +\epsilon\left\|\partial_tU_\ast(t)\right\|_{L^2\times L^2},\ \forall t\geq 0.
		\end{equation}
	\end{claim}
	
	With this result at hand, we do computations
	\begin{align*}
		\frac{d}{dt}\langle U_\ast,\partial_tU_\ast\rangle &= \left\|\partial_tU_\ast\right\|^2_{L^2\times L^2} +\langle U_\ast, \Delta U_\ast-U_\ast+\mathcal{N}(U_\ast)\rangle\\
		&= \left\|\partial_tU_\ast\right\|^2_{L^2\times L^2} -\left\|\nabla U_\ast\right\|^2_{L^2\times L^2}-\left\|U_\ast\right\|^2_{L^2\times L^2}+ \mathrm{Pot}_{\beta}(U_\ast) \\
		&\geq  \frac{1}{2}\left\|\partial_tU_\ast\right\|^2_{L^2\times L^2}-C\left\|\nabla U_\ast\right\|^2_{L^2\times L^2}
	\end{align*}
	where in the second equality we used integration by parts and definition of the potential energy, in the last inequality we used
	 \eqref{eq:11170} with $\epsilon=\frac{1}{2}$ and dropped the potential energy (since it is positive). Combining this last inequality with \eqref{eq:11170} with $\epsilon=\frac{1}{4}$, we obtain
	 \begin{equation}\label{eq:11171}
	 	\frac{1}{4} \left\|\partial_tU_\ast\right\|^2_{L^2\times L^2} +\left\|U_\ast\right\|^2_{\mathcal{H}\times\mathcal{H}}\leq C\left\|\nabla U_\ast\right\|^2_{L^2\times L^2}+ \frac{d}{dt}\langle U_\ast,\partial_tU_\ast\rangle.
	 \end{equation}
	Due to $K[U_\ast]\geq 0$, the left hand side $\approx E[\vec{U}_\ast]$. Thus integrating \eqref{eq:11171} from $0$ to $t_0$ gives
	\begin{equation}\label{eq:11172}
		t_0E[\vec{U}_\ast]\lesssim \int_0^{t_0}\left\|\nabla U_\ast(t)\right\|^2_{L^2\times L^2}dt
	\end{equation}
	
	We next introduce a bump function $\chi$ satisfying $\chi(x)=1$ for $|x|\leq 1$	and $\chi(x)=0$ for $|x|\geq 2$. Let $R>0$ be a number to be chosen later on. {A direct calculation yields}
	\begin{equation}\label{eq:11173}
		\frac{d}{dt}\left\langle \chi\left(\frac{\cdot}{R}\right) \partial_tU_\ast,\frac{1}{2}(x\cdot\nabla+\nabla\cdot x)U_\ast\right\rangle=-K_2[U_\ast]+\mathcal{O}\left(\int_{|x|>R}\left[|\nabla_{t,x}U_\ast|^2+|U_\ast|^2\right]\right)
	\end{equation}
	By Lemma \ref{lem:cond:2}, there exists a positive number $\delta_2$ so that
	\begin{equation*}
		K_2[U_\ast]\geq \delta_2\left\|\nabla U_\ast(t)\right\|^2_{L^2\times L^2},\ \ \forall t\geq 0.
	\end{equation*} Thanks to the precompactness of $U_\ast$, we take $R$ very large so that the second term on the right hand side of \eqref{eq:11173} is $<\delta_3 E[\vec{U}_\ast]$ for some positve $\delta_3\ll \delta_2$. With both results at hand, we obtain by integrating \eqref{eq:11173} in time from $0$ to $t_0$
	\begin{equation*}
		\left\langle \chi \partial_tU_\ast,\frac{1}{2}(x\cdot\nabla+\nabla\cdot x)U_\ast\right\rangle\bigg|_0^{t_0}\leq -\delta_2\int_0^{t_0}\left\|\nabla U_\ast(t)\right\|^2_{L^2\times L^2}dt+ Ct_0\delta_3E[\vec{U}_{\ast}].
	\end{equation*}
	Substituting \eqref{eq:11172} into this inequality, we get
	\begin{equation}\label{eq:11174}
		\left\langle \chi \partial_tU_\ast,\frac{1}{2}(x\cdot\nabla+\nabla\cdot x)U_\ast\right\rangle\bigg|_0^{t_0}\leq (C\delta_3-C_1\delta_2)t_0E[\vec{U}_\ast]
	\end{equation}
	Taking $t_0$ to be as in Lemma \ref{lem:0517} in this last inequality leads to
	\begin{equation}\label{eq:0517}
		\left\langle \chi \partial_tU_\ast,\frac{1}{2}(x\cdot\nabla+\nabla\cdot x)U_\ast\right\rangle\bigg|_0^{t_0}\leq - C\frac{\delta_2}{\epsilon}RE[\vec{U}_{\ast}]
	\end{equation}
	Noting that
	\begin{equation}
		\left|\mathrm{LHS}\eqref{eq:0517}\right| \lesssim RE[\vec{U}_{\ast}]
	\end{equation}
	we see that \eqref{eq:0517} is impossible by taking $\epsilon$ sufficiently small. This last contradiction shows that our initial {assumption} is false and hence finishes the proof of our main result.

\bibliographystyle{plain}
\bibliography{coeff}	

\begin{thebibliography}{10}

\bibitem{Adams-Fournier03}
R.~A. Adams and J.~F. Fournier.
\newblock {\em Sobolev spaces}, volume 140 of {\em Pure Appl. Math., Academic
  Press}.
\newblock New York, NY: Academic Press, 2nd ed. edition, 2003.

\bibitem{BCD11}
H.~Bahouri, J.-Y. Chemin, and R.~Danchin.
\newblock {\em Fourier analysis and nonlinear partial differential equations},
  volume 343 of {\em Grundlehren Math. Wiss.}
\newblock Berlin: Heidelberg, 2011.

\bibitem{Bahouri1999}
H.~Bahouri and P.~G\'{e}rard.
\newblock High frequency approximation of solutions to critical nonlinear wave
  equations.
\newblock {\em Amer. J. Math.}, 121(1):131--175, 1999.

\bibitem{chenzou2012}
Z.~{Chen} and W.~{Zou}.
\newblock {Positive least energy solutions and phase separation for coupled
  Schr\"odinger equations with critical exponent}.
\newblock {\em {Arch. Ration. Mech. Anal.}}, 205(2):515--551, 2012.

\bibitem{chenzou2015}
Z.~{Chen} and W.~{Zou}.
\newblock {Existence and symmetry of positive ground states for a doubly
  critical Schr\"odinger system}.
\newblock {\em {Trans. Am. Math. Soc.}}, 367(5):3599--3646, 2015.

\bibitem{Cui2022}
Y.~Cui and B.~Xia.
\newblock Criteria for finite time blow up for a system of
  klein{\textendash}gordon equations.
\newblock {\em Applicable Analysis}, pages 1--28, jul 2023.

\bibitem{Dodson2018}
B.~Dodson and J.~Murphy.
\newblock A new proof of scattering below the ground state for the non-radial
  focusing {NLS}.
\newblock {\em Math. Res. Lett.}, 25(6):1805--1825, 2018.

\bibitem{georgiev1990}
V.~Georgiev.
\newblock Global solution of the system of wave and {K}lein-{G}ordon equations.
\newblock {\em Math. Z.}, 203(4):683--698, 1990.

\bibitem{Gerard1998}
P.~G{\'{e}}rard.
\newblock Description of the lack of compactness for the sobolev imbedding.
\newblock {\em {ESAIM}: Control, Optimisation and Calculus of Variations},
  3:213--233, 1998.

\bibitem{ionescu2019}
A.~D. Ionescu and B.~Pausader.
\newblock On the global regularity for a wave-{K}lein-{G}ordon coupled system.
\newblock {\em Acta Math. Sin. (Engl. Ser.)}, 35(6):933--986, 2019.

\bibitem{Kenig2008}
C.~E. Kenig and F.~Merle.
\newblock Global well-posedness, scattering and blow-up for the energy-critical
  focusing non-linear wave equation.
\newblock {\em Acta Mathematica}, 201(2):147--212, 2008.

\bibitem{2007Coupled}
K.~R. Khusnutdinova.
\newblock Coupled klein–gordon equations and energy exchange in two-component
  systems.
\newblock {\em European Physical Journal Special Topics}, 147(1):45--72, 2007.

\bibitem{Masaki2023}
S.~Masaki and R.~Tsukuda.
\newblock Scattering below ground states for a class of systems of nonlinear
  schrodinger equations.
\newblock March 2023.

\bibitem{nakanishischlag2011}
K.~Nakanishi and W.~Schlag.
\newblock {\em Invariant manifolds and dispersive {H}amiltonian evolution
  equations}.
\newblock Zurich Lectures in Advanced Mathematics. European Mathematical
  Society (EMS), Z\"{u}rich, 2011.

\bibitem{pengwangwang2019}
S.~{Peng}, Q.~{Wang}, and Z.-Q. {Wang}.
\newblock {On coupled nonlinear Schr\"odinger systems with mixed couplings}.
\newblock {\em {Trans. Am. Math. Soc.}}, 371(11):7559--7583, 2019.

\bibitem{segal1965}
I.~E. Segal.
\newblock Nonlinear partial differential equations in quantum field theory.
\newblock In {\em Proc. {S}ympos. {A}ppl. {M}ath., {V}ol. {XVII}}, pages
  210--226. Amer. Math. Soc., Providence, R.I., 1965.

\bibitem{sirakov07}
B.~Sirakov.
\newblock Least energy solitary waves for a system of nonlinear schrdinger
  equations in $\mathbb{R}^n$.
\newblock {\em Communications in Mathematical Physics}, 271(1):199--221, 2007.

\bibitem{soavetavares2016}
N.~{Soave} and H.~{Tavares}.
\newblock {New existence and symmetry results for least energy positive
  solutions of Schr\"odinger systems with mixed competition and cooperation
  terms}.
\newblock {\em {J. Differ. Equations}}, 261(1):505--537, 2016.

\bibitem{tavares2011}
H.~{Tavares}, S.~{Terracini}, G.~{Verzini}, and T.~{Weth}.
\newblock {Existence and nonexistence of entire solutions for non-cooperative
  cubic elliptic systems}.
\newblock {\em {Commun. Partial Differ. Equations}}, 36(10-12):1988--2010,
  2011.

\bibitem{weiwu2020}
J.~{Wei} and Y.~{Wu}.
\newblock {Ground states of nonlinear Schr\"odinger systems with mixed
  couplings}.
\newblock {\em {J. Math. Pures Appl. (9)}}, 141:50--88, 2020.

\bibitem{Xia2019}
S.~Xia and C.~Xu.
\newblock On dynamics of the system of two coupled nonlinear schrödinger in.
\newblock {\em Mathematical Methods in the Applied Sciences},
  42(18):7096--7112, August 2019.

\bibitem{xu2014}
G.~{Xu}.
\newblock {Dynamics of some coupled nonlinear Schr\"odinger systems in
  \(\mathbb{R}^3\)}.
\newblock {\em {Math. Methods Appl. Sci.}}, 37(17):2746--2771, 2014.

\end{thebibliography}
\end{document}